\documentclass{article}

\usepackage{amsmath,amssymb,amsfonts,amsthm,mathtools,bm}
\usepackage{array,booktabs,longtable,multicol,multirow,makecell,tabularx,threeparttable}
\usepackage{enumitem}
\setlist[enumerate]{noitemsep, topsep=2pt}
\setlist[itemize]{noitemsep, topsep=2pt}

\usepackage[margin=0.9in]{geometry}
\usepackage[numbers,sort&compress]{natbib}
\usepackage[colorlinks,citecolor=purple,linkcolor=blue]{hyperref}
\usepackage{cleveref}
\usepackage{float}
\usepackage{caption,subcaption}
\usepackage{appendix}
\usepackage{authblk}
\usepackage{parskip}

\usepackage{algorithm,algorithmic}
\usepackage{eqparbox}

\usepackage{newfloat}
\newfloat{subroutine}{htbp}{los}
\floatname{subroutine}{Subroutine}
\crefname{subroutine}{subroutine}{subroutines}
\Crefname{subroutine}{Subroutine}{Subroutines}
\crefname{line}{line}{lines}

\Crefname{algorithm}{Alg.}{Alg.}

\newtheorem{theorem}{Theorem}[section]
\newtheorem{lemma}[theorem]{Lemma}

\newtheorem{corollary}[theorem]{Corollary}
\newtheorem{assumption}[theorem]{Assumption}

\newtheorem{definition}[theorem]{Definition}

\numberwithin{equation}{section}

\usepackage{tikz,pgfplots}
\usetikzlibrary{arrows.meta,calc,backgrounds,positioning}
\usepgfplotslibrary{fillbetween,patchplots}
\pgfplotsset{compat=newest}

\definecolor{ao(english)}{rgb}{0.0,0.5,0.0}
\definecolor{cadmiumgreen}{rgb}{0.0,0.42,0.24}
\definecolor{darkpastelgreen}{rgb}{0.01,0.75,0.24}

\tikzset{
  materia/.style={draw, fill=blue!20, text width=6em, text centered, minimum height=1.5em, drop shadow},
  etape/.style={materia, text width=8em, minimum width=10em, minimum height=3em, rounded corners, drop shadow},
  linepart/.style={draw, thick, color=black!50, -Latex, dashed},
  line/.style={draw, thick, color=black!50, -Latex},
  back group/.style={fill=white!20,rounded corners, draw=black!50, dashed, inner xsep=15pt, inner ysep=10pt},
}

\newcommand{\arc}{\texttt{CubicReg}}
\newcommand{\arca}{\texttt{CubicReg-A}}
\newcommand{\utr}{\texttt{UTR}}

\newcommand{\atr}{\Cref{alg.accelerated utr}}
\newcommand{\atrms}{\Cref{alg.ms accelerated utr}}
\newcommand{\tr}{\textrm{TR}}
\newcommand{\trp}{\textrm{TR}_+}
\numberwithin{equation}{section}

\author[1]{\small Yuntian Jiang}
\author[2]{\small Chuwen Zhang}
\author[1]{\small Bo Jiang\thanks{Corresponding author: isyebojiang@gmail.com}}
\author[3]{\small Yinyu Ye}

\affil[1]{\footnotesize School of Information Management and Engineering\\ Shanghai University of Finance and Economics}
\affil[2]{\footnotesize Booth School of Business\\ University of Chicago}
\affil[3]{\footnotesize Antai College of Economics and Management\\ Shanghai Jiao Tong University}

\title{Accelerating Trust-Region Methods: An Attempt to Balance Global and Local Efficiency}

\begin{document}

\maketitle
\begin{abstract}
    Balancing global efficiency and local convergence remains a central challenge in second-order methods for unconstrained convex optimization problems. Newton’s method enjoys fast local convergence but may diverge when initialized far from the solution~\cite{more1982newton,dennisjr1996numerical}. In contrast, accelerated second-order methods provide global guarantees but typically suffer from slower local convergence~\cite{carmon2022optimal,chen2022accelerating,jiang2020unified}. This raises the fundamental question of to what extent global acceleration can be achieved without sacrificing strong local convergence. In this paper, we tackle this challenge by proposing the first accelerated trust-region-type methods and leveraging their inherent primal-dual information. Our primary contribution is \emph{Accelerated Trust-Region method with Local Detection}, which utilizes the Lagrange multiplier to detect local regions and achieves a global oracle complexity of $\tilde{O}(\epsilon^{-1/3})$, while maintaining quadratic local convergence. We further examine the trade-off that arises when global convergence is pushed to the limit. Specifically, we introduce the \emph{Accelerated Trust-Region Extragradient Method}, which achieves a global oracle complexity of $\tilde{O}(\epsilon^{-2/7})$ but no longer enjoys quadratic local convergence. This establishes a phase-transition-like in accelerated trust-region-type methods: quadratic local convergence is preserved under moderate global acceleration, but it breaks down when pursuing extreme global efficiency. Numerical experiments are consistent with the theoretical predictions and illustrate the global-local trade-off.
\end{abstract}

\section{Introduction}\label{Intro}
In this paper, we consider the unconstrained convex optimization problem
\begin{equation} \label{eq.main problem}
    f^* = \min_{x\in \mathbb{R}^n} f(x),
\end{equation}
where $f:\mathbb{R}^n \to \mathbb{R}$ is a convex function.
When the Hessian of the objective is available, second-order methods (SOMs) are particularly effective for solving~\eqref{eq.main problem}, especially when high-accuracy solutions are desired. Their power stems from a central property: Newton’s method achieves quadratic convergence in a neighborhood of any nondegenerate minimizer~\cite{polyak2007newton,more1982newton,nocedal1999numerical}. Many SOMs are thus designed to approximate the Newton direction, allowing them to inherit its excellent local convergence behavior.

In addition to local convergence, another important criterion for evaluating the effectiveness of optimization methods is their non-asymptotic global oracle complexity, namely, the number of oracle calls (i.e., second-order subproblems) required to obtain an approximate solution~\cite{cartis2022evaluation}.  Over the past decades, a few SOMs have been developed with such emphases ~\cite{nesterov2006cubic,mishchenko2023regularized,jiang2026beyond,he2025homogeneous,curtis2021trust,he2023newton,cartis2011adaptive,cartis2011adaptive_i,curtis2017trust}.
In the convex setting, several of these methods admit further improvements.

In this context, two major approaches have emerged. The first is the accelerated cubic regularized Newton (CRN) method~\cite{nesterov2008accelerating}, which builds upon the cubic regularization oracle~\cite{nesterov2006cubic} and achieves a global oracle complexity of $O(\epsilon^{-1/3})$.
The second is the accelerated Newton proximal extragradient (A-NPE) method~\cite{monteiro2013accelerated}, which attains an oracle complexity of $\tilde{O}(\epsilon^{-2/7})$~ using second-order oracles satisfying certain error-bound conditions.
These developments are largely built upon the estimating sequence technique~\cite{nesterov1983method,baes2009estimate} and have inspired a wide range of enhancements~\cite{jiang2021optimal,chen2022accelerating,lin2022explicit,agafonov2023inexact,agafonov2024advancing,jiang2024accelerated,huang2024inexact,doikov2020contracting}.
More recently, the logarithmic factor in the oracle complexity of A-NPE was eliminated, yielding the optimal complexity of $O(\epsilon^{-2/7})$~\cite{carmon2022optimal,kovalev2022first}. In fact, these acceleration techniques are broadly applicable and are by no means limited to the domain of SOMs~\cite{xu2017accelerated,xu2018accelerated,lan2012optimal,ghadimi2016accelerated}.

However, it is challenging to balance the global and local efficiency of second-order optimization algorithms. While the classical Newton's method possesses quadratic local convergence, it lacks global guarantees and often diverges when the starting point is far from the optimal solution~\cite{more1982newton,dennisjr1996numerical}. In contrast, stronger worst-case complexity results of accelerated SOMs often come at the cost of local efficiency.
Empirical evidence shows that in tasks such as logistic regression, accelerated SOMs are often outperformed by unaccelerated ones~\cite{jiang2020unified,huang2024inexact,chen2022accelerating,carmon2022optimal}, and practical implementations often switch to Newton-type steps to obtain highly accurate solutions~\cite{jiang2020unified,huang2024inexact,chen2022accelerating}. The above phenomenon underscores a fundamental dilemma in accelerated SOMs, and raises the following question:
\begin{center}
    \emph{How much can we accelerate SOMs globally while maintaining their excellent local convergence?}
\end{center}

\subsection{Motivation and our approach}\label{sec.motivation}
It is well understood that the strong local performance of SOMs ultimately stems from the Newton step. Thus, to design an accelerated method with excellent local performance, the second-order oracle should have the ability to detect the local geometry near the optimal solution, providing a signal indicating that the iterates may have
entered the quadratic convergence region, allowing a Newton step to be active.
For this purpose, one particular choice could be the trust-region (TR) oracle~\cite{more1983recent,conn2000trust}. At a given point \( x \in \mathbb{R}^n \), the classic TR oracle solves the following subproblem:
\begin{equation}
    \label{eq.classic tr subproblem}
    \begin{aligned}
        \underset{d \in \mathbb{R}^n}{\min}~ & \nabla f(x)^T d + \tfrac{1}{2} d^T \nabla^2 f(x)  d \\
        \text{s.t.}                          & \quad  \|d\| \le r,
    \end{aligned}
\end{equation}
where \( r \) denotes the trust-region radius. We denote $(d,\lambda) = \tr(x, r)$ as the primal-dual solution to this problem. For an overview of TR methods, we refer readers to two excellent monographs \cite{yuan2000review,yuan2015recent}.

Two features make \eqref{eq.classic tr subproblem} particularly appealing. Near a non-degenerate minimizer, the step reduces to the pure Newton direction once the ball constraint becomes inactive, preserving quadratic convergence. This inactive-constraint mechanism provides a direct way to recover the pure Newton direction, a feature that is particularly convenient for local detection.
On the practical side, the TR oracle has been thoroughly analyzed \cite{jiang2022holderian,rojas2001new,honguyen2017second,wang2022generalized,so2007on,so2011deterministic}, so that highly competitive solvers have been developed \cite{rojas2001new, adachi2017solving, adachi2019eigenvalue} for it and its extensions.
To this end, TR oracles have formed the algorithmic core of general-purpose nonlinear programming solvers such as Knitro~\cite{byrd2006knitro}, COPT~\cite{ge2022cardinal}, and PDFO~\cite{ragonneau2024pdfo}. Similar success can be found in other cross-cutting applications ranging from adversarial training~\cite{yao2019trust} to reinforcement learning~\cite{schulman2015trust}. These advantages serve as the motivation to use \tr{} to address the global-local balance in accelerated SOMs.

However, to our best knowledge,  TR methods have not been accelerated yet as
using TR oracles introduces technical challenges. Classical TR methods cannot match the global oracle complexity of other unaccelerated SOMs~\cite{cartis2010on,curtis2017trust} because the dual variable $\lambda$ associated with the ball constraint in \eqref{eq.classic tr subproblem} is determined a posteriori and requires careful monitoring~\cite{curtis2018concise}. Achieving global acceleration demands relative stability between this multiplier and the step size, which roughly means that $\lambda$ should not vary too fast relative to the step size, but enforcing such stability can interfere with transitions to Newton steps near the optimum.

To address this challenge, we design, for the first time, the accelerated trust-region methods that leverage a modified TR oracle, which we call the trust-region oracle for acceleration denoted as \ref{eq.newTR}. In particular, at a point $x \in \mathbb{R}^n$, the \ref{eq.newTR} oracle solves
\begin{equation*}
    \begin{aligned}
        \min_{d \in \mathbb{R}^n}~ & \nabla f(x)^\top d
        + \tfrac{1}{2} d^\top (\nabla^2 f(x) + \sigma I) d, \\
        \text{s.t.}                & ~ \|d\| \le r.
    \end{aligned}
    \tag*{(TR$_+$)}
\end{equation*}
The dual variable associated with the ball constraint is denoted by $\lambda$, which can detect local geometric information of the optimal solution.
The primal regularization $\sigma$ in \ref{eq.newTR} is used to effectively modulate $\lambda$, and maintains the stability required for global acceleration, enabling our accelerated trust-region method to systematically balance fast local convergence with provable global guarantee.
\subsection{Contribution}
This work provides an answer to the question posed earlier: we demonstrate that the compatibility between the global and local efficiency of the SOMs depends on the degree of global acceleration.
Specifically, based on the \ref{eq.newTR} oracles \cite{jiang2026beyond}, we develop two accelerated TR-type methods for problem~\eqref{eq.main problem}. One can simultaneously achieve both global acceleration and quadratic local convergence, and the other has a faster global convergence rate but loses local efficiency.

The first method, \emph{Accelerated Trust-Region Method with Local Detection} (\Cref{alg.accelerated utr}), implements a local detection mechanism (\Cref{alg.local detection}) that automatically identifies and ``dives into'' the quadratic convergence regions. This design yields an improved global worst-case oracle complexity of $\tilde{O}(\epsilon^{-1/3})$, while achieving a local quadratic rate of convergence. To our best knowledge, this is the first accelerated SOM that provably achieves both $\tilde{O}(\epsilon^{-1/3})$ global oracle complexity and local quadratic convergence.

The second method, \emph{Accelerated Trust-Region Extragradient Method} (\Cref{alg.ms accelerated utr}), is designed to pursuit the limits of global performance, reaching the near-optimal global oracle complexity of $\tilde{O}(\epsilon^{-2/7})$, but does not retain the same local quadratic guarantee. Thus, our analysis and experiments reveal a phase-transition-like phenomenon in accelerated trust-region type methods: pushing global efficiency to its limits naturally entails a trade-off with local performance.
This behavior is caused by the fact that, in \Cref{alg.ms accelerated utr}, the primal regularizer must be selected through an extragradient search coupled with the extrapolation point, which prevents the simple $\lambda=0$ local-detection mechanism used in \Cref{alg.accelerated utr}.

Our theoretical findings are corroborated by numerical experiments. On a global scale, the accelerated trust-region methods and the accelerated cubic regularized Newton method~\cite{nesterov2008accelerating} consistently outperform the non-accelerated counterparts~\cite{jiang2026beyond,nesterov2006cubic}. Locally, \Cref{alg.accelerated utr} exhibits quadratic convergence, and \atrms{} cannot retain superlinear convergence near the solution, similar to the empirical conclusions in \cite{carmon2022optimal,jiang2024accelerated,huang2024inexact}. These findings confirm the necessity of local detection for retaining local quadratic convergence in TR-type methods.

\subsection{Related works}
We review related works along two main lines: TR-type methods and accelerated SOMs that aim to achieve faster-than-sublinear efficiency. Despite the strong empirical success of TR-type methods, their theoretical analysis mainly focuses on nonconvex problems~\cite{ye2005second,curtis2017trust,curtis2021trust,zhang2025homogeneous,hamad2022consistently,grapiglia2015on,hamad2024simple,curtis2023worst_i,ouyang2025trust}, while their global oracle complexity analysis for convex problems remains incomplete.
In fact, in convex optimization, the classical TR has worse convergence rates than other unaccelerated mainstream SOMs. This gap was recently closed in~\cite{jiang2026beyond}. However, whether TR oracles can further benefit from acceleration, as achieved for other SOMs, remains an open question.

As discussed earlier, unaccelerated SOMs often exhibit superior local performance near the optimal solution compared with their accelerated counterparts. A common remedy for this issue is to employ restart strategies in acceleration frameworks~\cite{nesterov2018lectures}. However, this introduces a new challenge of deciding when to restart since the optimal restart frequency depends on unknown, problem-specific parameters. In contrast, when stronger global regularity conditions (akin to strong convexity) are imposed on the objective function, the restart schedule can be determined in a more structured form~\cite{ostroukhov2020tensor, lin2025perseus, huang2025approximation} or even becomes unnecessary~\cite{jiang2022generalized, marquesalves2022variants}. Under such assumptions, these methods can achieve strong global performance, and thus the intrinsic trade-off between global convergence and local efficiency becomes a secondary concern.

\paragraph{Organization of the paper}
The remaining part of this paper is organized as follows. Section~\ref{sec.priliminary} reviews the background on TR-type oracles and the use of estimate sequence techniques in accelerated second-order optimization. In Section~\ref{sec.first variant}, we present the \Cref{alg.accelerated utr}, which achieves an oracle complexity of $\tilde O(\epsilon^{-1/3})$ and exhibits quadratic convergence in high-accuracy regimes. Section~\ref{sec.second variant} introduces the \Cref{alg.ms accelerated utr}, which explores the phase transition of acceleration and attains a near-optimal oracle complexity of $\tilde O(\epsilon^{-2/7})$. Section~\ref{sec.experiments} reports numerical results on practical tasks, confirming the theoretical advantages of our proposed methods.

\section{Preliminaries}
\label{sec.priliminary}
Throughout the paper, $\|\cdot\|$ denotes the standard Euclidean norm in $\mathbb{R}^n$. We use $x^T y$ and $\langle x,y\rangle$ interchangeably to denote the inner product between $x,y\in\mathbb{R}^n$. For a matrix $X\in\mathbb{R}^{n\times n}$, $\|X\|$ denotes the induced $\ell_2$ norm. Unless otherwise stated, $\log(\cdot)$ denotes the base-$2$ logarithm, while $\ln(\cdot)$ denotes the natural logarithm. In the statements of technical results, symbols such as $x_k,v_k,y_k$, etc., refer to the iterates and steps generated by the algorithm, rather than to generic vectors.

We are interested in finding approximate solutions of \eqref{eq.main problem} defined as below, which is also used in \citet{monteiro2013accelerated}.
\begin{definition}
    \label{def.approximate solution}
    Let $0<\epsilon_f<1$ and $0<\epsilon_g<1$ be two prescribed tolerances. A point $x\in\mathbb{R}^n$ is called an
    $\epsilon_f$-function-value solution of problem~\eqref{eq.main problem} if
    \begin{equation}\label{eq.function-value-sol}
        f(x)-f^* \leq \epsilon_f.
    \end{equation}
    A point $x\in\mathbb{R}^n$ is called an
    $\epsilon_g$-stationary-solution of problem~\eqref{eq.main problem} if
    \begin{equation}\label{eq.first-order-stat}
        \|\nabla f(x)\| \leq \epsilon_g.
    \end{equation}
    A point $x\in\mathbb{R}^n$ is called an
    $(\epsilon_f,\epsilon_g)$-approximate solution of problem~\eqref{eq.main problem} if it satisfies \eqref{eq.function-value-sol} or \eqref{eq.first-order-stat}.
\end{definition}

The following assumptions are used throughout the paper. Some of them are invoked only when necessary and will be stated explicitly. Note that Assumption~\ref{assm.bounded hessian} is equivalent to the gradient of the objective function being Lipschitz.
\begin{assumption}
    \label{assm.lipschitz}
    The objective function $f: \mathbb{R}^n \to \mathbb{R}$ is twice continuously differentiable, and its Hessian is Lipschitz continuous. That is, there exists a constant $M > 0$ such that for all $x, y \in \mathbb{R}^n$,
    \begin{equation}
        \label{eq.main assm}
        \|\nabla^2 f(x) - \nabla^2 f(y)\| \leq M \|x - y\|. \end{equation}
\end{assumption}
\begin{assumption}
    \label{assm.solvable}
    The problem \eqref{eq.main problem} is solvable; that is, there exists $x^* \in \mathbb{R}^n$ such that
    \begin{equation*}
        f(x^*) = f^*:= \min_{x \in \mathbb{R}^n} f(x).
    \end{equation*}
\end{assumption}
\begin{assumption}
    \label{assm.bounded hessian}
    The Hessian of the objective function is bounded. That is, there exists a constant $\kappa_H > 0$ such that for all $x \in \mathbb{R}^n$, \begin{equation}
        \label{eq.bounded hessian}
        \|\nabla^2 f(x)\| \leq \kappa_H.
    \end{equation}
\end{assumption}
As a direct consequence of Assumption~\ref{assm.lipschitz}, we have
\begin{lemma}[Lemma 4.1.1, \citet{nesterov2018lectures}]\label{lem.lipschitz}
    If \(f:\mathbb{R}^n \mapsto \mathbb{R}\) satisfies Assumption~\ref{assm.lipschitz}, then for all \(x,y\in \mathbb{R}^n\), we have
    \begin{subequations}
        \begin{align}
            \label{eq.first-order exp}
             & \left\|\nabla f(y)-\nabla f(x)-\nabla^{2} f(x)(y-x)\right\|                       \leq \frac{M}{2}\|y-x\|^{2}  \\
            \label{eq.second-order exp}
             & \left|f(y)-f(x)-\nabla f(x)^T(y-x)-\frac{1}{2}(y-x)^T\nabla^{2} f(x)(y-x)\right|  \leq \frac{M}{6}\|y-x\|^{3}.
        \end{align}
    \end{subequations}
\end{lemma}

To proceed, let us introduce a unified view of the updates of SOMs. In fact, almost all of the SOMs choose the step $d$ at a given point $x$ as
\begin{equation}
    \label{eq.second-order update}
    (\nabla^2 f(x) + \mu I )d = -\nabla f(x).
\end{equation}
The regularization parameter $\mu$ plays a critical role in the convergence analysis and varies among different SOMs. For example, in the classical TR method using the standard oracle~\eqref{eq.classic tr subproblem}, $\mu$ is fully determined as a posteriori, and is exactly the Lagrangian multiplier associated with the ball constraint in \eqref{eq.classic tr subproblem}. In the regularized Newton methods~\cite{mishchenko2023regularized}, $\mu$ can be explicitly selected as $\Theta(\sqrt{\|\nabla f(x)\|})$ such that an $O(\epsilon_f^{-1/2})$ method for finding $\epsilon_f$-function-value solutions can be designed. In the CRN method~\cite{nesterov2006cubic}, it requires some search procedures on $\mu$.

The following lemmas clarify $\mu$ plays an important role in the convergence analysis of SOMs, the proofs of which can be found in Appendix~\ref{sec.proof pre}.
\begin{lemma}
    \label{lem.bound next gradient}
    Suppose Assumption~\ref{assm.lipschitz} and \eqref{eq.second-order update} hold, then we have
    \begin{equation}
        \label{eq.bound next gradient}
        \|\nabla f(x+d)\| \leq \frac{M}{2}\|d\|^2+\mu\|d\|.
    \end{equation}
\end{lemma}

\begin{lemma}
    \label{lem.inner product g and d}
    Suppose Assumption~\ref{assm.lipschitz} holds and that \eqref{eq.second-order update} hold for some $\mu\geq M\|d\|$, then we have
    \begin{equation}
        \label{eq.inner product g d}
        \langle \nabla f(x+d),-d \rangle \geq \frac{\|\nabla f(x+d)\|^2}{2\mu}+\frac{3}{8}\mu \|d\|^2.
    \end{equation}
    Further, if $\mu \leq 2M\|d\|$, we have
    \begin{equation}
        \label{eq.inner product g d advanced}
        \langle \nabla f(x+d),-d \rangle \geq \frac{\sqrt{6}}{6} \frac{\|\nabla f(x+d)\|^{3/2}}{\sqrt{M}}.
    \end{equation}
\end{lemma}

Generally speaking, \Cref{lem.bound next gradient} shows that we can bound the next-iterate gradient norm by controlling $\mu$.
    {\Cref{lem.inner product g and d} serves as a key bridge to establish the global efficiency~(in fact, it is a modified version of~\cite[Lemma 4.2.5]{nesterov2018lectures} and~\cite[Corollary 1]{nesterov2021implementable} for TR-type oracle)}, which suggests that $\mu$ should be selected roughly as $\Theta(\|d\|)$. This condition generally fails for the classical TR oracle, as explained in~\Cref{sec.motivation} and also discussed in~\cite{curtis2018concise,curtis2017trust}. For this reason, we adopt the \ref{eq.newTR} oracle:
\begin{equation}\label{eq.newTR}
    \tag{TR$_+$}
    \begin{aligned}
        \min_{d \in \mathbb{R}^n}~ & \nabla f(x)^T d
        + \tfrac{1}{2} d^T (\nabla^2 f(x) + \sigma I) d, \\
        \text{s.t.}                & ~ \|d\| \le r,
    \end{aligned}
\end{equation}

We denote $(d, \lambda) = \trp(x, \sigma, r)$ as the primal-dual solution of the \ref{eq.newTR} subproblem, and in the following, all references to solution pairs $(d,\lambda)$ will refer to this \ref{eq.newTR} solution unless otherwise stated.
The global optimality conditions are as follows.
\begin{lemma}[Section 3,~\citet{conn2000trust}]\label{lem.optimal condition}
    The direction $d$ is the solution to \ref{eq.newTR} if and only if there exists a dual multiplier $\lambda \geq 0$ such that: \begin{subequations}\label{eq.sa}
        \begin{align}
            \label{eq.optcond primal}
             & \| d\| \leq r                                            \\
            \label{eq.optcond coml slack}
             & \lambda \left(\|d\|-r \right)=0                          \\ \label{eq.optcond firstorder}
             & (\nabla^2 f(x) + \sigma I + \lambda I ) d = -\nabla f(x) \\
            \label{eq.optcond secondorder}
             & \nabla^2 f(x) + \sigma I + \lambda I \succeq 0.
        \end{align}
    \end{subequations}
\end{lemma}
In the \ref{eq.newTR} oracle, the regularization parameter decomposes as $\mu=\sigma+\lambda$, aggregating \emph{a prior} regularizer from the primal problem and \emph{a posterior} dual variable. As we will show in the following, this primal-dual combo of information has the potential to balance global and local behavior.
\section{Variant~I: Balancing the Global-Local Trade-off}\label{sec.first variant}
In this section, we propose the first accelerated TR-type method, which applies the primal-dual structure of \ref{eq.newTR} into the estimating sequence \cite{nesterov2008accelerating}. The global oracle complexity for finding $(\epsilon_f,\epsilon_g)$-approximate solution is improved to $\tilde O(\epsilon_f^{-1/3})$ as opposed to the recent $O(\epsilon_f^{-1/2})$ in convex optimization \cite{jiang2026beyond}, while the local rate of convergence to the non-degenerate solution remains quadratic.
\subsection{Algorithm design}
\Cref{alg.accelerated utr} and its local detection mechanism (\Cref{alg.local detection}) are presented below.
\begin{algorithm}[ht]
    \caption{Accelerated Trust-Region Method with Local Detection (\atr{})}\label{alg.accelerated utr}
    \begin{algorithmic}[1]
        \STATE \textbf{input:} initial point $x_0=v_0\in \mathbb{R}^n$, $s_0=0$, stationarity tolerance $\epsilon_g>0$
        \FOR{$k=0, 1, 2, \ldots$}
        \STATE $y_k = \frac{k}{k+3}x_k+\frac{3}{k+3}v_k$
        \STATE $(\sigma_k,r_k) = (\tfrac{\sqrt{2M}}{2}\|\nabla f(y_k)\|^{1/2},\tfrac{1}{\sqrt{2M}}\|\nabla f(y_k)\|^{1/2})$
        \label{line.set-params}
        \STATE $(d_k,\lambda_k) = \trp(y_k,\sigma_k,r_k)$
        \STATE $x_{k+1} = y_k+d_k$
        \IF{$\lambda_k=0$} \label{line.if}
        \STATE \texttt{\# Enter Local Detection (\Cref{alg.local detection})}
        \STATE $(x_{k+1},\text{ET}) = \operatorname{LD}(y_k,\tfrac{\sqrt{2M}}{2}\|\nabla f(y_k)\|^{1/2},d_k,\epsilon_g)$
        \label{line.early terminate main}
        \IF{$\text{ET}$}
        \STATE \textbf{terminate and output $x_{k+1}$}
        \ENDIF
        \ENDIF
        \STATE $s_{k+1} = s_{k}+\tfrac{(k+1)(k+2)}{2}\nabla f(x_{k+1})$
        \label{line.update-s}
        \STATE $v_{k+1} = v_0-\sqrt{\tfrac{1}{24M\|s_{k+1}\|}}s_{k+1}$
        \label{line.update-v}
        \ENDFOR
    \end{algorithmic}
\end{algorithm}
\begin{subroutine}[!ht]
    \small
    \caption{Local Detection (LD) }\label{alg.local detection}
    \begin{algorithmic}[1]
        \STATE{\textbf{input:} $z_0=y\in \mathbb{R}^n$, $\mu_+$, $d_+$, $\epsilon_g$;
            \label{line.right bracket point}
        }
        \STATE{\texttt{\# Track 0: Stationary check}}
        \IF{$\|d_+\| \leq \min\{\frac{\epsilon_g}{2\kappa_H}, \sqrt{\frac{\epsilon_g}{M}}\}$}\label{line.if coro 3.1}
        \STATE{\textbf{output} $y+d_+$, ET =True
            \label{line.et 2}
        }
        \ENDIF
        \STATE{\texttt{\# Track 1: Local Diving}}
        \FOR{$i=1,\ldots,\left \lceil\log\frac{\|\nabla f(y)\|}{\epsilon_g} \right\rceil$}\label{line.track 1 number of iter}
        \STATE{$z_i = z_{i-1}-\nabla^2f(z_{i-1})^{-1}\nabla f(z_{i-1})$}
        \IF{$\|\nabla f(z_i)\|\leq \epsilon_g$}
        \STATE{\textbf{output} $z_i$, ET=True
            \label{line.et 1}
        }
        \ENDIF
        \ENDFOR
        \STATE{\texttt{\# Track 2: Ratio Bracketing and Bisection~(R\&B)}}
        \IF{$\frac{\mu_+}{\|d_+\|}\leq 2M$
            \label{line.check right bracket point}}
        \STATE{\textbf{output} $y+d_+$, ET =False
            \hfill\texttt{\# Check if the bisection is needed}
        }
        \ELSE
        \STATE{$r_- = \|d_+\|$, $r_+ =  \|(\nabla^2f(y)+M\|d_+\| I)^{-1}\nabla f(y)\|$
            \label{line.left bracketing point}
        }
        \ENDIF
        \WHILE{$\frac{\mu}{\|d\|}<M$ or $\frac{\mu}{\|d\|}>2M$}
        \STATE{$r = \tfrac{r_-+r_+}{2}$
            \hfill\texttt{\# Perform bisection over $r \in [r_-, r_+]$}
        }
        \STATE{$(d,\mu) = \trp(y,0,r)$
            \label{line.linear system oracle}
        }
        \IF{$\frac{\mu}{\|d\|}<M$}
        \STATE{$r_+=r$}
        \ELSIF{$\tfrac{\mu}{\|d\|}>2M$}
        \STATE{$r_- = r$}
        \ENDIF
        \ENDWHILE
        \STATE{\textbf{output} $y + d$, ET= False}
    \end{algorithmic}
    \normalsize
\end{subroutine}
The overall regularization in \ref{eq.newTR} is a combination of primal and dual information: $\mu_k = \sigma_k + \lambda_k$. The primal part $\sigma_k$ is chosen carefully to make the dual part more stable and still carry local information, while the dual part $\lambda_k$ is used as an indicator to check whether the current iterate is close to the local optimum.

The case $\lambda_k > 0$ typically indicates that the current iterate is still far from a local optimum. In this regime, as shown later in \Cref{lem.ratio pass nonzero}, the overall regularization parameter $\mu_k$ becomes automatically proportional to the step size, which justifies the updates of the estimating sequence in the acceleration framework.

The case $\lambda_k = 0$ is a signal that the current extrapolation point $y_k$ may have entered the local quadratic convergence region of Newton's method. This triggers the Local Detection procedure (LD, \Cref{alg.local detection}), which consists of three separate tracks named stationary check, local diving, and ratio bracketing and bisection (R\&B), respectively, probing the local geometry around $y_k$.

In \Cref{alg.local detection}, stationary check ({Track 0}) is a simple examination to check if the current step is small enough, which justifies an $\epsilon_g$-stationary-solution (see \Cref{coro.second et}).
In local diving ({Track 1}), we run Newton's method (equivalently, \ref{eq.newTR}$(x, 0, r)$ with sufficiently large $r$) starting from $y_k$, the resulting sequence is indexed by $i$ and denoted by $\{z_i\}$. As discussed in Section~\ref{sec.local}, if $y_k$ indeed lies within the quadratic convergence region, the resulting sequence $\{z_i\}$ will exhibit quadratic convergence, which justifies early termination of \Cref{alg.accelerated utr}. Indeed, we could stop diving if the $\|\nabla f(z_i)\|$ is increasing, or a degenerate Hessian is detected.

In R\&B ({Track 2}), we attempt to find a step $d_k$ whose size is proportional to the regularization parameter $\mu_k$. This step certifies the updates of the estimate sequence and thus preserves global acceleration.

In summary, \Cref{alg.accelerated utr} admits two possible termination routes: when $y_k$ enters the region of local quadratic convergence, the algorithm detects this and terminates early with ET$=$True~(see lines~\ref{line.et 2} and \ref{line.et 1} in \Cref{alg.local detection}). Otherwise, it continues to follow a path of globally accelerated sequences with ET$=$False.

\subsection{Global convergence in function value}
Now we analyze the worst-case global oracle complexity of \Cref{alg.accelerated utr} for computing an $(\epsilon_f,\epsilon_g)$-approximate solution as defined in \Cref{def.approximate solution}. Since our goal here is to analyze the worst-case global oracle complexity, we assume the iterates do not step into local diving of \Cref{alg.local detection}, which is more relevant to local convergence analysis (cf. \Cref{sec.local}). Our analysis follows a standard procedure: we first establish the \textit{iteration complexity} of the accelerated sequence in~\Cref{alg.accelerated utr}, and subsequently determine the number of oracle calls required per iteration in~\Cref{alg.local detection}, thereby obtaining the final \textit{oracle complexity} bound.
\subsubsection{Iterations complexity of \Cref{alg.accelerated utr}}
In this section, we show that \Cref{alg.accelerated utr} takes $O(\epsilon_f^{-1/3})$ to find an $\epsilon_f$-function-value solution as in \Cref{def.approximate solution}.

When using the estimating sequence technique in~\cite{nesterov2008accelerating}, the globally accelerated convergence is guaranteed by maintaining the following two relations across iterations
\begin{subequations}\label{eq.es relation}
    \begin{align}
        \label{eq.es relation 1}
         & \phi_k^*  \geq A_k f(x_k), \quad  \phi_k^*= \phi_k(v_k) =\min_{x\in\mathbb{R}^n}\phi_k(x) \\
        \label{eq.es relation 2}
         & \phi_k(x)\leq A_kf(x)+\phi_0(x), \quad \forall x\in\mathbb{R}^n.
    \end{align}
\end{subequations}
Here $\{\phi_k(x)\}_{k\geq 0}$ is a sequence of functions that approximate $f(x)$ from with $\{v_k\}_{k\geq0}$ being the optimum, $\{A_k\}_{k\geq 0}$ is a sequence that measures the convergence rate of the sequence $\{x_k\}_{k\geq 0}$. As a result of \eqref{eq.es relation}, we have
\begin{equation*}
    f(x_k)-f(x)\leq \frac{1}{A_k}\phi_0(x),\quad \forall x\in \mathbb{R}^n, \  k\geq 1.
\end{equation*}
Therefore, the global iteration complexity directly follows from the choice of $A_k$ and $\phi_0(x)$.

For the upcoming analysis, we first complete the definition of $A_k,\phi_k(x)$ in the estimating sequence framework that was not explicitly presented in the algorithm:
\begin{equation}
    \label{eq.update a and A}
    a_k = \frac{(k+1)(k+2)}{2}, \quad A_k=\frac{k(k+1)(k+2)}{6},
\end{equation}
\begin{align}\label{eq.update phi}
    \phi_0(x) = 8M\|x-x_0\|^3, \ \phi_{k+1}(x) =\phi_k(x)+a_k\left (f(x_{k+1})+\langle \nabla f(x_{k+1}),x-x_{k+1}\rangle \right ).
\end{align}
We introduce some basic properties of the estimating sequence in \Cref{sec.proof pre}, which are from \citet{nesterov2018lectures}. To proceed, we categorize the iterates generated by \Cref{alg.accelerated utr} into two disjoint sets based on the value of the multiplier, according to whether the iterations invoke local detection:
\begin{equation}
    \label{eq.outer index catagory}
    \mathcal{Z} = \left \{k\in \mathbb{N} | \lambda_k=0\right\}, \quad \mathcal{N} = \left \{ k\in\mathbb{N} |  \lambda_k>0\right\}.
\end{equation}
We first demonstrate that the ratio between the regularizer and the step size remains stable for $k\in\mathcal{N}$ as shown in \eqref{eq.output of alg2} which exhibits the power of \ref{eq.newTR}.
\begin{lemma}
    \label{lem.ratio pass nonzero}
    In the $k$-th iteration of \Cref{alg.accelerated utr}, if $k\in\mathcal{N}$, then
    \begin{equation}
        \label{eq.output of alg2}
        d_k = -(\nabla^2 f(y_k)+\mu_k I )^{-1}\nabla f(y_k), \quad
        M\|d_k\|\leq \mu_k \leq 2M\|d_k\|.
    \end{equation}
\end{lemma}
\begin{proof}
    On the one hand, by our choice of $(\sigma_k,r_k)$ in Line~\ref{line.set-params} of \Cref{alg.accelerated utr}, we have $$\mu_k=\sigma_k+\lambda_k\geq\sigma_k= Mr_k= M\|d_k\|.$$
    On the other hand, since $\nabla^2 f(y_k)\succeq 0$, we have
    \begin{equation*}
        \mu_k\|d_k\|=\|\mu_kd_k\|\leq \|(\nabla^2 f(y_k)+\mu_k I)d_k\|=\|\nabla f(y_k)\|.
    \end{equation*}
    Note that $\lambda_k>0$ implies $\|d_k\|=r_k=\frac{1}{\sqrt{2M}}\|\nabla f(y_k)\|^{1/2}$. As a result, we have
    \begin{equation*}
        \mu_k\leq \frac{\|\nabla f(y_k)\|}{\|d_k\|}=\sqrt{2M}\|\nabla f(y_k)\|^{1/2}=2M\|d_k\|.
    \end{equation*}
    Therefore, \eqref{eq.output of alg2} holds.
\end{proof}
To maintain the flow of the complexity analysis, we temporarily assume the output of R\&B of \Cref{alg.local detection} satisfies~\eqref{eq.output of alg2} as well, i.e., \eqref{eq.output of alg2} holds for $k\in\mathcal{Z}$, whenever ET=False (we will verify this in \Cref{lem.rboutput}). Now we show that the relation \eqref{eq.es relation} holds during the update of \Cref{alg.accelerated utr}.
\begin{lemma}\label{lem.estimate sequence}
    Suppose Assumption~\ref{assm.lipschitz} holds, and \Cref{eq.output of alg2} holds for the output of \Cref{alg.local detection}. Then we can guarantee the following for all $k\geq 0$:
    \begin{equation}
        \label{eq.estimating sequence relation}
        A_k f(x_k)\leq \phi_k^*  \leq \phi_k(x)\leq A_k f(x)+\phi_0(x), \quad \forall x\in \mathbb{R}^n.
    \end{equation}
\end{lemma}
\begin{proof}
    We conduct the proof by induction. Note that $A_0 = 0$, so \eqref{eq.estimating sequence relation} holds for $i=0$. Suppose it also holds for $i=k$.
    We first check \eqref{eq.es relation 2} for $i=k+1$:
    \begin{align*}
        \phi_{k+1}(x) & =_{(a)}\phi_k(x)+a_k\left(f(x_{k+1})+\langle \nabla f(x_{k+1}),x-x_{k+1}\rangle \right)               \\
                      & \leq_{(b)} A_k f(x)+\phi_0(x) +a_k\left(f(x_{k+1})+\langle \nabla f(x_{k+1}),x-x_{k+1}\rangle \right) \\
                      & \leq_{(c)} A_k f(x)+\phi_0(x) +a_k f(x)                                                               \\
                      & =_{(d)}A_{k+1} f(x)+\phi_0(x),
    \end{align*}
    where $(a)$ is from \eqref{eq.update phi}, $(b)$ is from the induction hypothesis, $(c)$ is from the convexity of $f$, and $(d)$ is from \eqref{eq.update a and A}.
    We now check \eqref{eq.es relation 1} for $i=k+1$,
    \begin{align*}
        \phi_{k+1}^*
         & = \min_{x\in\mathbb{R}^n} \left\{ \phi_k(x) + a_k \left(f(x_{k+1}) + \langle \nabla f(x_{k+1}), x - x_{k+1} \rangle \right) \right\}                                                                                        \\
         & \geq_{(a)} \min_{x\in\mathbb{R}^n} \Bigl\{ \phi_k^* + 4M \|x - v_k\|^3
        + a_k \left(f(x_{k+1}) + \langle \nabla f(x_{k+1}), x - x_{k+1} \rangle \right) \Bigr\}                                                                                                                                        \\
         & \geq_{(b)} \min_{x\in\mathbb{R}^n} \Bigl\{ A_k f(x_k) + 4M \|x - v_k\|^3
        + a_k \left(f(x_{k+1}) + \langle \nabla f(x_{k+1}), x - x_{k+1} \rangle \right) \Bigr\}                                                                                                                                        \\
         & \geq_{(c)} \min_{x\in\mathbb{R}^n} \Bigl\{ A_k f(x_{k+1}) + A_k \langle \nabla f(x_{k+1}), x_k - x_{k+1} \rangle + 4M \|x - v_k\|^3 + a_k \left(f(x_{k+1}) + \langle \nabla f(x_{k+1}), x - x_{k+1} \rangle \right) \Bigr\} \\
         & =_{(d)} A_{k+1} f(x_{k+1})
        + A_{k+1} \left\langle \nabla f(x_{k+1}),
        \frac{A_k}{A_{k+1}} x_k + \frac{a_k}{A_{k+1}} v_k - x_{k+1} \right\rangle - \frac{a_k^{3/2}}{3\sqrt{3} \cdot \sqrt{M}} \|\nabla f(x_{k+1})\|^{3/2}                                                                             \\
         & =_{(e)} A_{k+1} f(x_{k+1}) + A_{k+1} \langle \nabla f(x_{k+1}), y_k - x_{k+1} \rangle
        - \frac{a_k^{3/2}}{3\sqrt{3} \cdot \sqrt{M}} \|\nabla f(x_{k+1})\|^{3/2}                                                                                                                                                       \\
         & = A_{k+1} f(x_{k+1}) + A_{k+1} \left(
        \langle \nabla f(x_{k+1}), y_k - x_{k+1} \rangle
        - A_{k+1}^{-1} \cdot \frac{a_k^{3/2}}{3\sqrt{3} \cdot \sqrt{M}}
        \|\nabla f(x_{k+1})\|^{3/2} \right)                                                                                                                                                                                            \\
         & \geq_{(f)} A_{k+1} f(x_{k+1}) + A_{k+1} \left(
        \langle \nabla f(x_{k+1}), y_k - x_{k+1} \rangle
        - \frac{\sqrt{6}}{6\sqrt{M}} \|\nabla f(x_{k+1})\|^{3/2} \right).
    \end{align*}
    In the above analysis, $(a)$ is from \eqref{eq.phi_k phi_low}, $(b)$ is from the induction hypothesis, $(c)$ is from the convexity of $f$, $(d)$ is from
    \begin{equation*}
        \begin{aligned}
             & \min_{x} \Bigl\{ A_k f(x_{k+1}) + A_k \langle \nabla f(x_{k+1}), x_k - x_{k+1} \rangle + 4M \|x - v_k\|^3 + a_k \left(f(x_{k+1}) + \langle \nabla f(x_{k+1}), (x - v_k) + v_k - x_{k+1} \rangle \right) \Bigr\} \\
             & \quad = A_{k+1} f(x_{k+1}) + \left\langle \nabla f(x_{k+1}), A_k x_k - A_{k+1} x_{k+1} + a_k v_k \right\rangle + \min_{y} \left\{ 4M \|y\|^3 + \langle a_k \nabla f(x_{k+1}), y \rangle \right\}                \\
             & \quad = A_{k+1} f(x_{k+1}) + A_{k+1} \left\langle \nabla f(x_{k+1}), \frac{A_k}{A_{k+1}} x_k + \frac{a_k}{A_{k+1}} v_k - x_{k+1} \right\rangle - \frac{1}{3\sqrt{3M}} \|a_k \nabla f(x_{k+1})\|^{3/2};
        \end{aligned}
    \end{equation*}
    $(e)$ is from \eqref{eq.update a and A}, and $(f)$ is from \eqref{eq.magnitude a and A}. Now it reduces to proving
    \begin{equation*}
        \langle \nabla f(x_{k+1}),y_k-x_{k+1} \rangle -\frac{\sqrt{6}}{6\sqrt{M}}\|\nabla f(x_{k+1})\|^{3/2} \geq 0.
    \end{equation*}
    Using \Cref{lem.ratio pass nonzero},
    \begin{equation*}
        M\|d_k\| \leq \mu_k \leq 2M\|d_k\|.
    \end{equation*}
    As a result of \eqref{eq.inner product g d advanced}, it holds that
    \begin{equation*}
        \langle \nabla f(x_{k+1}),y_k-x_{k+1} \rangle -\frac{\sqrt{6}}{6\sqrt{M}}\|\nabla f(x_{k+1})\|^{3/2} \geq 0.
    \end{equation*}
    Thus the proof is complete.
\end{proof}
We have the following global iteration complexity result of \Cref{alg.accelerated utr}.
\begin{theorem}
    \label{thm.convergence outer loop}
    Suppose Assumption~\ref{assm.lipschitz} and \ref{assm.solvable}
    hold, and suppose that \Cref{eq.output of alg2} is satisfied at every non-terminated
    iteration. Then for every generated iterate $x_k$, $k\geq 1$, we have
    \[
        f(x_k)-f^*
        \leq
        \frac{48MD_0^3}{k(k+1)(k+2)}, \quad D_0 = \|x_0-x^*\|.
    \]
\end{theorem}
\begin{proof}
    By \Cref{lem.estimate sequence} with $x=x^*$, for every $k\ge1$, $f(x_k)-f^*\le \phi_0(x^*)/A_k=8M\|x_0-x^*\|^3/(k(k+1)(k+2)/6).$
\end{proof}
\begin{corollary}
    \label{coro.function value complexity}
    Suppose Assumption~\ref{assm.lipschitz} and Assumption~\ref{assm.solvable}
    hold, and suppose that the ratio condition required by
    \Cref{lem.estimate sequence} is satisfied at every non-terminated
    iteration. Then it takes \Cref{alg.accelerated utr} $O\left(\frac{M^{1/3}D_0}{\epsilon_f^{1/3}}\right)$ iterations to find $\epsilon_f$-function-value solutions as in \Cref{def.approximate solution}.
\end{corollary}
\subsubsection{Number of oracle calls of \Cref{alg.local detection}}\label{sec.line search}
We now estimate the number of oracle calls in \Cref{alg.local detection}.
As the argument is quite technical, some of the proofs are postponed to Appendix~\ref{sec.proof variant 1}.
In this part of the analysis, we may assume that $\|d_+\| > \min\{\frac{\epsilon_g}{2\kappa_H},\sqrt{\frac{\epsilon_g}{M}}\}$; otherwise, it ends at {Track 0}. Furthermore, since we also limit the calls in {Track~1}, we only have to estimate number of oracle calls in Track 2 (R\&B), i.e., the complexity of bisection over radius $r$ on the interval $[r_-,r_+]$. 
We will show that the number of \ref{eq.newTR} oracles needed is bounded above by $O\left(\log\left(1/{\epsilon_g}\right)\right)$.

In each step of the R\&B, the update $(d,\mu)=\trp(y,0,r)$ with $r\in[r_-,r_+]$ defines the correspondence between $r$ and $\mu$ over the interval $[r_-,r_+]$ and $[\mu_-,\mu_+]$, where $\mu_-$ and $\mu_+$ have the following value due to the mechanism of \Cref{alg.accelerated utr} and \Cref{alg.local detection}
\begin{equation}\label{eq.value mu}
    \mu_- = M\|d_+\|, \quad \mu_+=\frac{\sqrt{2M}}{2}\|\nabla f(y)\|^{1/2}.
\end{equation}
We show that this correspondence is one-to-one.

\begin{lemma}
    \label{lem.one-to-one correspondence}
    For any $r \in [r_-, r_+]$, there exists a unique $\mu \in [\mu_-, \mu_+]$ such that
    \[
        r = \|(\nabla^2 f(y) + \mu I)^{-1} \nabla f(y)\| := r_y(\mu),
    \]
    establishing a one-to-one correspondence between $r$ and $\mu$.
    In particular, we may write $\mu = r_y^{-1}(r)$,
    and the correspondence of endpoints satisfies
    \[
        r_- = r_y(\mu_+),
        \quad
        r_+ = r_y(\mu_-).
    \]
\end{lemma}
\begin{proof}
    The endpoints correspondence between $r_- \to \mu_+$ and $r_+ \to \mu_-$ follows from line~\ref{line.left bracketing point} of \Cref{alg.local detection} and \eqref{eq.value mu}.
    For any $r$, the existence of such $\mu$ follows from $(d,\mu)=\trp(y,0,r)$ and $\eqref{eq.optcond firstorder}$. For uniqueness, noting $\nabla^2 f(y) + \mu I \succ 0$ for any $\mu \in [\mu_-, \mu_+]$,
    we can apply the eigenvalue decomposition
    $\nabla^2 f(y) = V \Lambda V^\top$ with $\Lambda = \mathrm{diag}(\zeta_1,\dots,\zeta_n)$ and $V$ orthogonal.
    Then
    \[
        r = \bigl\|(\nabla^2 f(y) + \mu I)^{-1}\nabla f(y)\bigr\|
        = \sqrt{\sum_{i=1}^n \frac{\beta_i^2}{(\zeta_i + \mu)^2}},
    \]
    where $\beta_i = \nabla f(y)^\top v_i$. Since $\zeta_i \ge 0$ and $\mu > 0$, every term in the summation is continuous and strictly decreasing in $\mu$, and thus $r_y(\mu)$ is continuous and strictly decreasing on $[\mu_-, \mu_+]$.
    Consequently, for any $r \in [r_-, r_+]$, there exists a unique $\mu \in [\mu_-, \mu_+]$ such that
    \[
        r = r_y(\mu) = \bigl\|(\nabla^2 f(y) + \mu I)^{-1}\nabla f(y)\bigr\|.
    \]
    Hence, the correspondence between $r$ and $\mu$ is one-to-one, establishing the desired relation $\mu = r_y^{-1}(r)$.
\end{proof}
Using this one-to-one correspondence, when we perform bisection on $r$ over the interval $[r_-,r_+]$, we also implicitly perform bisection on $\mu$ over $[\mu_-,\mu_+]$. The logic of our analyses is to first identify the target interval of $\mu$ and then use the one-to-one correspondence again to transfer it to the target interval of $r$. This observation motivates us to introduce the following auxiliary functions of $\mu$.
\begin{equation}
    \label{eq.auxiliary function for bisection}
    g_y(\mu) =\frac{\mu}{r_y(\mu)}=\frac{\mu}{\| (\nabla^2 f(y)+\mu I)^{-1}\nabla f(y)\|}.
\end{equation}
Thus, if ET=False, we have successfully located some $\mu$ such that
\begin{equation*}
    M \leq g_y(\mu) \leq 2M,
\end{equation*}
matching the condition needed by \Cref{thm.convergence outer loop}. For completeness, we state the following lemma and the proof is immediate.
\begin{lemma}\label{lem.rboutput}
    In \textnormal{R\&B} of \Cref{alg.local detection}, if ET=False, then \Cref{eq.output of alg2} holds.
\end{lemma}
The following lemma validates the choices of the bracketing points at line~\ref{line.left bracketing point} of \Cref{alg.local detection}.
\begin{lemma}
    \label{lem.valid bracketing}
    In \textnormal{R\&B} of \Cref{alg.local detection}, when the bisection begins, we have
    \begin{equation*}
        g_y(\mu_-) < M, \quad g_y(\mu_+) >2M.
    \end{equation*}
\end{lemma}
\begin{proof}
    From line~\ref{line.check right bracket point} in \Cref{alg.local detection}, we have $g_y(\mu_+)=\frac{\mu_+}{\|d_+\|} >2M$ whenever \Cref{alg.local detection} enters bisection procedure.

    Denote $d_- = -\left(\nabla^2 f(y)+\mu_-I\right)^{-1}\nabla f(y)$. Since $\mu_- <\mu_+$, then
    \begin{equation*}
        \|d_-\|=\| (\nabla ^2 f(y)+\mu_- I )^{-1} \nabla f(y)\|>\| (\nabla ^2 f(y)+\mu_+ I )^{-1} \nabla f(y)\| = \|d_+\|,
    \end{equation*}
    therefore $g_y(\mu_-) = \frac{\mu_-}{\|d_-\|}=\frac{M\|d_+\|}{\|d_-\|}< M$.
\end{proof}
Next, we identify the length of the target interval for $\mu$.
\begin{lemma}
    \label{lem.bound target interval}
    There exist $\mu_l,\mu_u$ with $\mu_-<\mu_l<\mu_u<\mu_+$ such that
    \begin{equation}\label{eq.target interval bracketing points}
        g_{y}(\mu_l)=M, \quad g_{y}(\mu_u)=2M,
    \end{equation}
    and for all $\mu\in[\mu_l,\mu_u]$, we have
    \begin{equation*}
        M\leq g_y(\mu) \leq2M.
    \end{equation*}
    The length of the target interval $[\mu_l,\mu_u]$ satisfies
    \begin{equation}
        \label{eq.bound target interval}
        \mu_u-\mu_l \geq \frac{M\| (\nabla ^2 f(y)+\mu_u I  )^{-1}\nabla f(y)\|}{2}.
    \end{equation}
\end{lemma}
After using the one-to-one correspondence to transfer the target interval to the one with respect to $r$, the complexity of R\&B is bounded as follows. The details can be found in the Appendix~\ref{sec.proof variant 1}.
\begin{lemma}
    \label{lem.bound initial interval}
    Suppose Assumption~\ref{assm.lipschitz} and Assumption~\ref{assm.bounded hessian} hold. At the $k$-th iteration of \Cref{alg.accelerated utr}, the oracle complexity of the R\&B procedure is
    \begin{equation}
        \label{eq.complexity bisection}
        O\left(\log\tfrac{\kappa_H+\mu_+}{\epsilon_g}\right).
    \end{equation}
\end{lemma}
To summarize, in \Cref{alg.local detection}, the number of oracle calls  reduces to
\begin{equation*}
    \max\left\{O\left(\log \tfrac{\|\nabla f(y)\|}{\epsilon_g}\right),
    O\left(\log\tfrac{\kappa_H+\mu_+}{\epsilon_g}\right)
    \right\}.
\end{equation*}
\subsubsection{Oracle complexity of \Cref{alg.accelerated utr}}\label{sec.oracle compl}
We begin with the following property of the iterates $y_k$ if $k\in\mathcal{Z}$, which invoked \Cref{alg.local detection}, cf. \eqref{eq.outer index catagory}.
\begin{lemma}
    \label{lem.upper bound for zero lambda}
    Suppose Assumption~\ref{assm.bounded hessian} hold, for $k\in \mathcal{Z}$, we have
    \begin{equation}
        \label{eq.upper bound for zero lambda}
        \|\nabla f(y_k)\|\leq \frac{2\kappa_H^2}{M}.
    \end{equation}
\end{lemma}
\begin{proof}
    Note that when $k\in \mathcal{Z}$, we have $\lambda_k =0$, also from Line~\ref{line.set-params} hence the following hold
    \begin{equation*}
        \left (\nabla^2 f(y_k)+\frac{\sqrt{2M}}{2}\|\nabla f(y_k)\|^{1/2} I\right )d_+ = -\nabla f(y_k), \ \|d_+\| \leq \frac{1}{\sqrt{2M}}\|\nabla f(y_k)\|^{1/2}.
    \end{equation*}
    Therefore, by the triangle inequality, it holds that
    \begin{equation*}
        \|\nabla f(y_k)\| \leq \|\nabla^2 f(y_k) d_+\| +\frac{\sqrt{2M}}{2}\|\nabla f(y_k)\|^{1/2}\|d_+\| \leq \|\nabla^2 f(y_k) d_+\|+\frac{1}{2}\|\nabla f(y_k)\|.
    \end{equation*}
    By Assumption~\ref{assm.bounded hessian}, and rearranging the terms we have
    \begin{equation*}
        \begin{aligned}
            \kappa_H \cdot \frac{1}{\sqrt{2M}}\|\nabla f(y_k)\|^{1/2}  \geq \kappa_H \|d_+\| & \geq \|\nabla^2 f(y_k) d_+\|       \geq \frac{1}{2}\|\nabla f(y_k)\|.
        \end{aligned}
    \end{equation*}
    By rearranging terms, we obtain \eqref{eq.upper bound for zero lambda}.
\end{proof}
Using the results above, we know that the output of {Track 0} is a $\epsilon_g$-stationary-solution, similar to Track 1.
\begin{corollary}
    \label{coro.second et}
    Suppose Assumption~\ref{assm.lipschitz} and \ref{assm.bounded hessian} hold,
    it holds that the output of {Track 0} (at line~\ref{line.if coro 3.1} of \Cref{alg.local detection}) satisfies
    \begin{equation*}
        \|\nabla f(y+d_+)\|\leq \epsilon_g.
    \end{equation*}
\end{corollary}
\begin{proof}
    From \Cref{lem.bound next gradient}, we have
    \begin{equation*}
        \|\nabla f(y+d_+)\| \leq_{(a)} \frac{M}{2}\|d_+\|^2 + \frac{\sqrt{2M}}{2}\|\nabla f(y)\|^{1/2}\|d_+\|   \leq \frac{M}{2}\times \frac{\epsilon_g}{M}+\frac{\sqrt{2M}}{2}\times\frac{\sqrt{2}\kappa_H}{\sqrt{M}}\times\frac{\epsilon_g}{2\kappa_H}  \leq \epsilon_g,
    \end{equation*}
    where $(a)$ is from $\lambda=0$ and Line~\ref{line.set-params}.
\end{proof}

Since $\mu_+ = \frac{\sqrt{2M}}{2}\|\nabla f(y)\|^{1/2}$ and we already established the boundedness for the gradient norm in \Cref{lem.upper bound for zero lambda}, we can derive the oracle complexity of \Cref{alg.accelerated utr}.
\begin{theorem}
    \label{thm.total complexity newton}
    Suppose Assumption~\ref{assm.lipschitz}, Assumption~\ref{assm.solvable} and Assumption~\ref{assm.bounded hessian} hold, it takes at most
    \begin{equation}\label{eq.total complexity newton}
        O\left (\epsilon_f^{-1/3}\log \left(1/\epsilon_g\right)\right)
    \end{equation}
    \ref{eq.newTR} oracles for \Cref{alg.accelerated utr} to find an $(\epsilon_f,\epsilon_g)$-approximate solution as in \Cref{def.approximate solution}.
\end{theorem}
\begin{proof}
    This result is a direct combination of Theorem~\ref{thm.convergence outer loop} with \Cref{lem.bound initial interval} and \Cref{lem.upper bound for zero lambda}.
\end{proof}
\subsection{Local convergence rate of \Cref{alg.accelerated utr}}
\label{sec.local}
In this section, we analyze the local convergence rate of \Cref{alg.accelerated utr}. Compared to other accelerated SOMs, \Cref{alg.accelerated utr} shows a quadratic local convergence in favor of the local detection mechanism. The quadratic local convergence is essential to find high-accuracy approximate solutions. First, we make the following standard assumption for local convergence analysis.
\begin{assumption}
    \label{assm.local strongly convex}
    Suppose problem \eqref{eq.main problem} has a unique solution $x^*$, which satisfies the second-order sufficient optimality condition
    \begin{equation}
        \label{eq.second-order suff cond}
        \nabla f(x^*)=0, \quad \nabla^2 f(x^*)\succeq \nu I
    \end{equation}
    for some $\nu >0$.
\end{assumption}
If the above assumption holds, it is a well-known result that Newton's method converges quadratically if initialized in a small region that contains the optimum.
\begin{lemma}[Theorem 1.2.5,~\cite{nesterov2018lectures}]
    \label{lem.quadratic point}
    Suppose Assumption~\ref{assm.lipschitz} and Assumption~\ref{assm.local strongly convex} hold and $\{z_i\}$ is the sequence generated by Newton's method. Denoting $R_i = \|z_i-x^*\|$, when the initial point $z_0$ is in the region:
    \begin{equation}
        \label{eq.local quadratic region}
        \mathcal{LQ}_p = \{x\in\mathbb{R}^n| \ \|x-x^*\| < \frac{2\nu}{3M}\},
    \end{equation}
    then $z_i \in \mathcal{LQ}_p$ for all $i$ and it converges quadratically to $x^*$:
    \begin{equation}
        \label{eq.quadratic convergence}
        R_{i+1} < \frac{3M}{2\nu}R_i^2<R_i.
    \end{equation}
\end{lemma}
The subscript $p$ is mnemonic for point-wise. Similar to the standard quadratic convergence region discussed above, the gradient norm also exhibits quadratic convergence when the initial point is sufficiently close to $x^*$. To establish this result, we first present the following technical lemma, which will be central to the subsequent analysis.
\begin{lemma}[Corollary 1.2.2,~\cite{nesterov2018lectures}]
    Suppose~Assumption~\ref{assm.lipschitz} and Assumption~\ref{assm.local strongly convex} hold, for any $x$ with $\|x-x^*\| \leq r$, we have
    \begin{equation}
        \label{eq.psd region}
        \nabla^2 f(x^*)-MrI\preceq \nabla^2 f(x) \preceq \nabla^2 f(x^*)+MrI.
    \end{equation}
    Therefore, when $x\in \mathcal{LQ}_p$, we have
    \begin{equation}
        \label{eq.psd xk}
        \nabla^2 f(x) \succ \frac{\nu}{3}I.
    \end{equation}
\end{lemma}
Now we present the region of quadratic convergence $\mathcal{LQ}_g$ for the norm of the gradient, the subscript $g$ is mnemonic for gradient-wise.
\begin{lemma}
    \label{lem.quadratic gradient}
    Suppose Assumption~\ref{assm.lipschitz} and Assumption~\ref{assm.local strongly convex} hold and $\{z_i\}$ is the sequence generated by Newton's method. When the initial point $z_0$ is in the region:
    \begin{equation}
        \label{eq.local quadratic region gradient}
        \mathcal{LQ}_g = \{x\in\mathbb{R}^n| \ \|x-x^*\| < \frac{2\nu}{3M}, \ \|\nabla f(x)\|< \frac{2\nu^2}{9M}\},
    \end{equation}
    then $z_i \in \mathcal{LQ}_g$ for all $i$ and the norm of the gradient $\|\nabla f(z_i)\|$ converges quadratically to $0$:
    \begin{equation}
        \label{eq.quadratic convergence gradient}
        \|\nabla f(z_{i+1})\| < \frac{9M}{2\nu^2}\|\nabla f(z_i)\|^2<\|\nabla f(z_i)\|.
    \end{equation}
\end{lemma}
\begin{proof}
    We prove by induction. Suppose $z_i\in \mathcal{LQ}_g$. From \eqref{eq.first-order exp}, we have
    \begin{equation*}
        \|\nabla f(z_{i+1}) -\nabla f(z_i)-\nabla^2 f(z_i)(z_{i+1}-z_i)\| \leq \frac{M}{2}\|z_{i+1}-z_i\|^2.
    \end{equation*}
    Noting that $z_{i+1} = z_i-\nabla^2 f(z_i)^{-1}\nabla f(z_i)$, we have
    \begin{equation*}
        \|\nabla f(z_{i+1})\| \leq \frac{M}{2}\|\nabla^2 f(z_i)^{-1} \nabla f(z_i)\|^2  \leq \frac{M}{2}\|\nabla^2 f(z_i)^{-1}\|^2\| \nabla f(z_i)\|^2 <_{(a)} \frac{9M}{2\nu^2}\|\nabla f(z_i)\|^2 <_{(b)} \|\nabla f(z_i)\|.
    \end{equation*}
    $(a)$ is from $R_i< \frac{2\nu}{3M}$ and \eqref{eq.psd xk}. $(b)$ is from the fact $z_i\in \mathcal{LQ}_g$. $R_{i+1} \leq \frac{2\nu}{3M}$ follows from \Cref{lem.quadratic point}. Therefore, $z_{i+1}\in \mathcal{LQ}_g$.
\end{proof}
As a result, we know how many iterations are needed to find a point with $\|\nabla f(z_i)\|\leq \epsilon_g$ if $z_0 \in \mathcal{LQ}_g$ and $\{z_i\}$ is generated by Newton's method initialized at $z_0$.
\begin{corollary}
    \label{lem.number of iterations newton}
    Suppose Assumption~\ref{assm.lipschitz} and Assumption~\ref{assm.local strongly convex} hold. Let $0<\epsilon_g< \frac{1}{C(M,\nu)}$, $z_0\in\mathcal{LQ}_g$ and $\{z_i\}$ generated by Newton's method. Then when
    \begin{equation}
        \label{eq.number of iterations newton}
        i\geq \left\lceil \log\left( \frac{\ln\left(\tfrac{1}{C(M,\nu)\epsilon_g}\right)}{\ln\left(\tfrac{1}{\eta_0}\right)} \right) \right\rceil \quad \text{with} \quad C(M,\nu) = \frac{9M}{2\nu^2}, \quad \eta_0 = \|\nabla f(z_0)\| C(M,\nu)<1,
    \end{equation}
    we have $\|\nabla f(z_i)\|\leq \epsilon_g$.
\end{corollary}
\begin{proof}
    From \eqref{eq.quadratic convergence gradient} we have
    \begin{equation*}
        \|\nabla f(z_i)\| \leq C(M,\nu)\|\nabla f(z_{i-1})\|^2 \leq C(M,\nu)^{2^{i-1}} \|\nabla f(z_0)\|^{2^i} \leq \frac{(C(M,\nu)\|\nabla f(z_0)\|)^{2^i}}{C(M,\nu)}.
    \end{equation*}
    Therefore, when $i$ hits the threshold defined in \eqref{eq.number of iterations newton}, we have $\|\nabla f(z_i)\|\leq \epsilon_g$.
\end{proof}
After all these preparations, we show that \Cref{alg.local detection} can guarantee that we can find an $(\epsilon_f,\epsilon_g)$-approximate solution with a local quadratic rate. The proof consists of the following three steps:
\begin{enumerate}[label=(\roman*)]
    \item After finite number of iterations of \Cref{alg.accelerated utr}, $y_k\in\mathcal{LQ}_g$;
    \item Whenever $y_k\in\mathcal{LQ}_g$, $\lambda_k=0$ in \Cref{alg.accelerated utr};
    \item Once $y_k\in \mathcal{LQ}_g$, the Local Diving procedure in \Cref{alg.local detection} converges quadratically with ET=True.
\end{enumerate}
First, we show the first step, i.e., $y_k$ will finally enter the region $\mathcal{LQ}_g$.
\begin{lemma}\label{lem.enter region}
    Suppose Assumptions~\ref{assm.lipschitz}, \ref{assm.bounded hessian},
    and \ref{assm.local strongly convex} hold. Then, whenever the $k$-th iterate of
    \Cref{alg.accelerated utr} is generated, if
    \[
        k\geq
        \left\lceil
        36\,\frac{\kappa_H M \|x_0-x^*\|}{\nu^2}
        \right\rceil,
    \]
    we have $y_k\in\mathcal{LQ}_g$.
\end{lemma}

The proof of the above lemma can be found in \Cref{sec.proof variant 1}.
Next, we prove the second step: in \Cref{alg.accelerated utr}, when $y_k\in\mathcal{LQ}_g$, we must have $\lambda_k=0$.
\begin{lemma}
    \label{lem.detection of local quadratic}
    Suppose Assumption~\ref{assm.lipschitz} and Assumption~\ref{assm.local strongly convex} hold. In \Cref{alg.accelerated utr}, if $y_k\in\mathcal{LQ}_g$, then $\lambda_k=0$.
\end{lemma}
\begin{proof}
    From \eqref{eq.optcond coml slack}, to prove $\lambda_k=0$, it suffices to prove that the step $d_k$ lies in the trust region. First from Line~\ref{line.set-params} in \Cref{alg.accelerated utr}
    \begin{equation*}
        \begin{aligned}
            \|(\nabla^2 f(y_k)+\sigma_k I)^{-1}\nabla f(y_k)\| & = \left\|\left (\nabla^2 f(y_k)+\frac{\sqrt{2M}}{2} \|\nabla f(y_k)\|^{1/2} I\right)^{-1}\nabla f(y_k)\right\|         \\
                                                               & \leq \left\|\left (\nabla^2 f(y_k)+\frac{\sqrt{2M}}{2} \|\nabla f(y_k)\|^{1/2} I\right)^{-1}\right\| \|\nabla f(y_k)\|
        \end{aligned}
    \end{equation*}
    Note $\nabla^2 f(y_k)\succ \frac{\nu}{3}$ as analyzed in \eqref{eq.psd xk}.
    \begin{equation*}
        \begin{aligned}
            \|(\nabla^2 f(y_k)+\sigma_k I)^{-1}\nabla f(y_k)\| & \leq \frac{\|\nabla f(y_k)\|}{\frac{\nu}{3}+\frac{\sqrt{2M}}{2} \|\nabla f(y_k)\|^{1/2}} \\&<_{(a)} \frac{\|\nabla f(y_k)\|}{\frac{\sqrt{2M}}{2}\|\nabla f(y_k)\|^{1/2}+\frac{\sqrt{2M}}{2} \|\nabla f(y_k)\|^{1/2}} =\frac{\|\nabla f(y_k)\|^{1/2}}{\sqrt{2M}} =r_k.
        \end{aligned}
    \end{equation*}
    $(a)$ is because of the definition of $\mathcal{LQ}_g$ in \eqref{eq.local quadratic region gradient}. Therefore, we conclude that the TR constraint is inactive and $\lambda_k=0$.
\end{proof}
Last, we show the third step: when the tolerance $\epsilon_g$ is small enough, \Cref{alg.accelerated utr} will be early terminated, i.e., we have ET$=$True when $y_k$ enters $\mathcal{LQ}_g$.
\begin{theorem}
    \label{thm.local quadratic convergence}
    Suppose Assumption~\ref{assm.lipschitz} and Assumption~\ref{assm.local strongly convex} hold, then there exists $\epsilon^* >0$ such that when
    \begin{equation}
        \label{eq.regularity epsilon}
        0<\epsilon_g< \min\{\epsilon^*, \frac{1}{C(M,\nu)}\},
    \end{equation}
    the $\epsilon_g$-stationary-solutions will be found in $O\left(\log\log\left(1/\epsilon_g\right)\right)$ iterations when $y_k$ enters $\mathcal{LQ}_g$, and ET$=$True in \Cref{alg.accelerated utr}.
\end{theorem}
\begin{proof}
    We first consider the case that ET$=$True at line~\ref{line.et 2} of \Cref{alg.local detection}, then $y_k+d_k$ is already an $\epsilon_g$-stationary-solution~(see \Cref{coro.second et}) and there is nothing to prove.

    Next, we consider the case \Cref{alg.local detection} enters local diving. We know from \eqref{eq.psd xk} that when $y_k\in\mathcal{LQ}_g$, $\|\nabla f(y_k)\|\leq \frac{1}{C(M,\nu)}$, then from Lemma~\ref{lem.loglog vs log}~(see Appendix~\ref{sec.proof variant 1}) we know that local diving will find $\epsilon_g$-stationary-solutions within the iteration number defined in \eqref{eq.number of iterations newton}, which is strictly less than the maximal number of iteration we defined at line~\ref{line.track 1 number of iter}, i.e., $\left \lceil\log\frac{\|\nabla f(y)\|}{\epsilon_g} \right\rceil$.
\end{proof}

\section{Variant~II: Pushing Global Efficiency to the Limit}\label{sec.second variant}
We now turn to another accelerated TR method that incorporates the acceleration framework of \citet{monteiro2013accelerated}. This method achieves a near-optimal global oracle complexity of $\tilde{O}(\epsilon_f^{-2/7})$ for finding $(\epsilon_f,\epsilon_g)$-approximate solutions. We refer to this approach as the Accelerated Trust-Region Extragradient Method (\Cref{alg.ms accelerated utr}).
\subsection{Algorithm design}
We begin by outlining the acceleration framework in \Cref{alg.ms accelerated utr}. The main technical ingredient is to integrate the \ref{eq.newTR} oracle into a modified version of the framework proposed by \citet{monteiro2013accelerated}, providing another eligible oracle choice besides the cubic regularization oracle in the previous literature \cite{carmon2022optimal}.

At the core of \Cref{alg.ms accelerated utr} lies an implicit search procedure (\Cref{alg.bisection ms}, we still call it R\&B for short) along the curve:
\begin{gather} \label{eq.update yk} y_k(\sigma) = \frac{A_k}{A_k+a_k(\sigma)}x_k+\frac{a_k(\sigma)}{A_k+a_k(\sigma)}v_k\\ \label{eq.update ak} a_k(\sigma) =\frac{1+\sqrt{1+4A_k\sigma}}{2\sigma}.
\end{gather}
In the whole process, R\&B~(\Cref{alg.bisection ms}) determines the primal regularization $\sigma_k$, and thus $a_k(\sigma_k)$ and the extrapolation point $y_k(\sigma_k)$, at which \ref{eq.newTR} is called.
\begin{algorithm}[ht]
    \caption{Accelerated Trust-Region Extragradient Method~(\atrms{})}\label{alg.ms accelerated utr}
    \begin{algorithmic}[1]
        \STATE{\textbf{input:} $x_0=v_0\in \mathbb{R}^n$, \ $A_0 =0$, \ R\&B search threshold $\theta>1$, \ $0<\eta<1$, damping parameter $\gamma \leq \frac{1}{\theta}$, stationarity tolerance $\epsilon_g>0$}
        \FOR{$k=0, 1, 2, \ldots$}
        \STATE{$(x_{k+1},\sigma_k,\text{ET}) = \operatorname{R\&B}(x_k,v_k,A_k,\eta,\theta,\epsilon_g)$}
        \IF{$\text{ET}$}
        \STATE \textbf{terminate and output $x_{k+1}$
            \label{line.early terminate ms}
        }
        \ENDIF
        \STATE{$a_k = a_k(\sigma_k)$}
        \STATE{
            $v_{k+1} = v_k-\gamma a_k \nabla f(x_{k+1})$}
        \STATE{
            $A_{k+1} = A_k+a_k$
            \label{line.alg2-update-A}
        }
        \ENDFOR
    \end{algorithmic}
\end{algorithm}
\begin{subroutine}[!ht]
    \caption{Ratio Bracketing and Bisection~(R\&B) for \Cref{alg.ms accelerated utr}}\label{alg.bisection ms}
    \begin{algorithmic}[1]
        \STATE{\textbf{input:} $x,v\in \mathbb{R}^n$, $A,0<\eta<1,\theta>1,\epsilon_g>0$;}
        \STATE{\texttt{\# Bracket points}}
        \STATE{Set $\sigma_- = \sqrt{\frac{2M\epsilon_g}{1+2\theta}}, \ \sigma_+=\sqrt{\frac{MG_0}{\eta}}$, $(d_-,\lambda_-) = \trp(y(\sigma_-),\sigma_-,\frac{1}{M}\sigma_-)$
            \label{line.def G0}
        }
        \IF{$\lambda_- \leq (\theta-1)\sigma_-$}\label{line.early termination criterion}
        \STATE{\textbf{Output} $y(\sigma_-)+d_-,\sigma_-,\text{ET=True}$}
        \label{line.early termination ms}
        \ELSE
        \STATE{\texttt{\# Perform bisection over $\sigma\in [\sigma_-,\sigma_+]$}}
        \WHILE{i) $\lambda = 0$ and $\|d\| < \frac{\eta}{M}\sigma$, or ii) $\lambda > (\theta-1)\sigma$}\label{line.goal of bisection ms}
        \STATE{$\sigma = \frac{\sigma_-+\sigma_+}{2}$}
        \STATE{$(d,\lambda) = \trp(y(\sigma),\sigma,\frac{1}{M}\sigma)$\label{line.call-utr-bisec}}
        \IF{i) holds}
        \STATE{$\sigma_+ =\sigma$}
        \ELSIF{ii) holds}
        \STATE{$\sigma_- =\sigma$}
        \ENDIF
        \ENDWHILE
        \STATE{\textbf{output} $y(\sigma)+d,\sigma,\text{ET=False}$}
        \ENDIF
    \end{algorithmic}
\end{subroutine}
In R\&B~(\Cref{alg.bisection ms}), the search task initializes the bracketing points at line~\ref{line.def G0}, where
$$G_0 = \max\left\{\|\nabla f(x)\| \mid \|x-x^*\|\leq \left(\frac{4}{\sqrt{3\gamma}}+1\right)D_0\right\}, \quad D_0 = \|x_0-x^*\|.$$
At line~\ref{line.goal of bisection ms}, we proceed bisection until $\sigma_k$ has been located such that
\begin{equation}
    \label{eq.relation lambda sigma}
    0 \leq \lambda_k \leq (\theta - 1)\sigma_k, \quad \|d_k\| \geq \frac{\eta}{M} \sigma_k,
\end{equation}
where $\theta$ and $\eta$ are prescribed thresholds for the regularization term and the step size, respectively. They will work together with the damping parameter $\gamma$ in \Cref{alg.ms accelerated utr} to guarantee acceleration for \ref{eq.newTR} oracle. Inequalities \eqref{eq.relation lambda sigma} serves a similar purpose as \eqref{lem.ratio pass nonzero} to \Cref{alg.accelerated utr}, which we will further clarify in \Cref{lem.reduce line search}.

Similar to \Cref{alg.accelerated utr}, there are two paths to terminate \Cref{alg.ms accelerated utr}. The first path occurs when the flag $\textrm{ET} = \text{True}$, indicating early termination (see line~\ref{line.early termination ms}). This happens when an $\epsilon_g$-stationary-solution is found in R\&B~(\Cref{alg.bisection ms}).
The second path arises if \Cref{alg.ms accelerated utr} is not early terminated. In this case, the step $d_k$ satisfies \eqref{eq.relation lambda sigma} for each $k$, and consequently, the estimating sequence guarantees that \Cref{alg.ms accelerated utr} outputs an $\epsilon_f$-function-value solution.
\subsection{Global convergence in function value}
The proof sketch is similar to the one in the previous section: we first establish \textit{iteration complexity} then move on to \textit{oracle complexity}.
\subsubsection{Iteration complexity of \Cref{alg.ms accelerated utr}}
We first show that \Cref{alg.ms accelerated utr} can terminate early in R\&B~(\Cref{alg.bisection ms}).
\begin{lemma}
    \label{lem.early termination ms}
    Suppose Assumption~\ref{assm.lipschitz} holds and \Cref{alg.bisection ms} outputs ET$=$True at line~\ref{line.early terminate ms}, we have
    \begin{equation*}
        \|\nabla f(x_{k+1}) \|\leq \epsilon_g.
    \end{equation*}
\end{lemma}
\begin{proof}
    By definition, if ET$=$True, then $x_{k+1} = y(\sigma_-)+d_-$. Therefore, we have
    by \Cref{lem.lipschitz}, we have
    \begin{equation*}
        \|\nabla f(y(\sigma_-)+d_-)-\nabla f(y(\sigma_-))-\nabla^2 f(y(\sigma_-))d_-\|\leq \frac{M}{2}\|d_-\|^2.
    \end{equation*}
    By triangle inequality we have
    \begin{equation*}
        \begin{aligned}
            \|\nabla f(y(\sigma_-)+d_-)\| & \leq \frac{M}{2}\|d_-\|^2+\| \nabla f(y(\sigma_-))+\nabla^2 f(y(\sigma_-))d_-\| \\
                                          & =_{(a)} \frac{M}{2}\|d_-\|^2+(\sigma_-+\lambda_-) \|d_-\|                       \\
                                          & \leq_{(b)} \frac{1+2\theta}{2M}\sigma_-^2  \leq_{(c)} \epsilon_g,
        \end{aligned}
    \end{equation*}
    where $(a)$ is from \eqref{eq.optcond firstorder}, $(b)$ is from Line~\ref{line.def G0} and thus $\|d_-\| \leq \frac{1}{M}\sigma_-$, $\sigma_-+\lambda_-\leq \theta\sigma_-$, and $(c)$ is from Line~\ref{line.def G0}.
\end{proof}
Since we focus on the worst-case global oracle complexity, we assume that early termination does not occur; that is, the ``If'' route at line~\ref{line.early termination criterion} of R\&B~(\Cref{alg.bisection ms}) is never triggered. Similar to previous analyses, we will show that \Cref{alg.ms accelerated utr} takes $O(\epsilon_f^{-2/7})$ to find an $\epsilon_f$-function-value approximate solution satisfying \Cref{def.approximate solution}.

First, we assume the bisection procedure is valid in the sense that \eqref{eq.relation lambda sigma} holds for each iteration~(which will be verified in~\Cref{sec.validty of bisection}).
As a result, the following recursive relation holds.
\begin{lemma}\label{lem.ms estimate}
    Suppose Assumption~\ref{assm.lipschitz}, Assumption~\ref{assm.solvable} hold, and the output of \textnormal{R\&B}~(\Cref{alg.bisection ms}) satisfies \eqref{eq.relation lambda sigma}, i.e.,
    \begin{equation*}
        0\leq\lambda_k\leq (\theta-1)\sigma_k, \quad \|d_k\|\geq \frac{\eta}{M}\sigma_k,
    \end{equation*}
    then in the $k$-th iteration of \Cref{alg.ms accelerated utr}, the following holds:
    \begin{equation}
        \label{eq.helper bounded x}
        \frac{1}{2}\|v_{k+1}-x^* \|^2 +\gamma A_{k+1}\left(f(x_{k+1})-f^*\right)+ \frac{3\gamma A_{k+1}\sigma_k}{8}\|d_k\|^2 \leq \frac{1}{2}\|v_k-x^*\|^2 +\gamma A_k \left (f(x_k)-f^*\right).
    \end{equation}
    Further, if the output of \textnormal{R\&B}~(\Cref{alg.bisection ms}) satisfies \eqref{eq.relation lambda sigma} for all iteration $i$ with $0\leq i\leq k$, then
    \begin{equation}
        \label{eq.ms estimate advanced}
        \gamma A_{k+1}\left (f(x_{k+1})-f^* \right) + \frac{1}{2}\|v_{k+1}-x^*\|^2 +B_{k+1} \leq \frac{1}{2}\|v_0-x^*\|^2,
    \end{equation}
    where $B_{k+1} = \frac{3\gamma M}{8}\sum_{i=0}^{k}A_{i+1}\|d_i\|^3.$
\end{lemma}
\begin{proof}
    Note the way we update $v_{k+1}$ in \Cref{alg.ms accelerated utr}.
    \begin{align*}
        \|v_{k+1} - x^*\|^2
         & = \|v_k - x^* - \gamma a_k \nabla f(x_{k+1})\|^2                                                                                                             \\
         & = \|v_k - x^*\|^2 + \gamma^2 a_k^2 \|\nabla f(x_{k+1})\|^2
        - 2\gamma a_k \langle \nabla f(x_{k+1}), v_k - x^* \rangle                                                                                                      \\
         & =_{(a)} \|v_k - x^*\|^2 + \gamma^2 a_k^2 \|\nabla f(x_{k+1})\|^2  - 2\gamma \langle \nabla f(x_{k+1}), (A_k + a_k) y_k(\sigma_k) - A_k x_k - a_k x^* \rangle \\
         & = \|v_k - x^*\|^2 + \gamma^2 a_k^2 \|\nabla f(x_{k+1})\|^2                                                                                                   \\
         & \quad - 2\gamma \langle \nabla f(x_{k+1}),
        (A_k + a_k)(y_k(\sigma_k) - x_{k+1}) + A_k(x_{k+1} - x_k)
        + a_k(x_{k+1} - x^*) \rangle                                                                                                                                    \\
         & \leq_{(b)} \|v_k - x^*\|^2 + \gamma^2 a_k^2 \|\nabla f(x_{k+1})\|^2 - 2\gamma (A_k + a_k) \langle \nabla f(x_{k+1}), y_k(\sigma_k) - x_{k+1} \rangle         \\
         & \quad - 2\gamma A_k \left(f(x_{k+1}) - f(x_k)\right)
        - 2\gamma a_k \left(f(x_{k+1}) - f^*\right).
    \end{align*}

    $(a)$ comes from \eqref{eq.update yk}. $(b)$ is from the convexity of $f$. Using $A_{k+1}=A_k+a_k$ as in \Cref{alg.ms accelerated utr} and rearranging terms gives
    \begin{align*}
             & \|v_{k+1}-x^*\|^2 +2\gamma A_{k+1} \left (f(x_{k+1})-f^*\right)                                                                                                                                                                          \\
        \leq & \|v_k-x^*\|^2+2\gamma A_k \left (f(x_k)-f^*\right )+\gamma^2 a_k^2 \|\nabla f(x_{k+1})\|^2 -2\gamma \left (A_k+a_k\right)\langle \nabla f(x_{k+1}), y_k(\sigma_k)-x_{k+1} \rangle                                                        \\
        \leq & _{(a)} \|v_k-x^*\|^2+2\gamma A_k \left (f(x_k)-f^*\right )+\gamma^2 a_k^2 \|\nabla f(x_{k+1})\|^2 -\gamma \left (A_k+a_k\right)\left (\frac{\|\nabla f(x_{k+1})\|^2}{\sigma_k+\lambda_k}+\frac{3(\sigma_k+\lambda_k)}{4}\|d_k\|^2\right) \\
        \leq & _{(b)} \|v_k-x^*\|^2+2\gamma A_k \left (f(x_k)-f^*\right )+\left(\gamma^2 a_k^2-\frac{\gamma(A_k+a_k)}{\theta \sigma_k}\right)\|\nabla f(x_{k+1})\|^2 -\frac{3\gamma(A_k+a_k)(\sigma_k+\lambda_k)}{4}\|d_k\|^2.
    \end{align*}
    $(a)$ is from whenever we call the oracle at Line~\ref{line.def G0} and \ref{line.call-utr-bisec}, we have $\sigma \geq M\|d\|$, hence we can apply \eqref{eq.inner product g d}. $(b)$ is from \eqref{eq.relation lambda sigma}. Note that $\gamma \leq \frac{1}{\theta}$ and \eqref{eq.update ak} hold. We have
    \begin{equation*}
        \gamma^2 a_k^2-\frac{\gamma(A_k+a_k)}{\theta \sigma_k} = \gamma\left(\gamma a_k^2-\frac{A_k+a_k}{\theta \sigma_k}\right)\leq \frac{\gamma}{\theta}\left(a_k^2-\frac{A_k+a_k}{\sigma_k}\right)=0,
    \end{equation*}
    which gives \eqref{eq.helper bounded x}.
    Also, because $\sigma_k \geq M\|d_k\|$ as stated,
    \begin{equation*}
        \frac{1}{2}\|v_{k+1}-x^* \|^2 +\gamma A_{k+1}\left (f(x_{k+1}-f^*)\right)+ \frac{3\gamma M A_{k+1}}{8}\|d_k\|^3 \leq \frac{1}{2}\|v_k-x^*\|^2 +\gamma A_k \left (f(x_k)-f^*\right).
    \end{equation*}
    Summing the above inequality over $i = 0, 1, \dots, k$, we obtain
    \begin{equation*}
        \sum_{i=0}^k \left( \frac{1}{2}\|v_{i+1}-x^* \|^2 +\gamma A_{i+1}\left(f(x_{i+1})-f^*\right) \right) + \frac{3\gamma M}{8} \sum_{i=0}^k A_{i+1}\|d_i\|^3 \leq\  \sum_{i=0}^k \left( \frac{1}{2}\|v_i-x^*\|^2 +\gamma A_i \left(f(x_i)-f^*\right) \right).
    \end{equation*}
    By canceling the common terms on both sides (telescoping sum), this simplifies to
    \begin{equation*}
        \frac{1}{2}\|v_{k+1}-x^* \|^2 +\gamma A_{k+1}\left(f(x_{k+1})-f^*\right) + \frac{3\gamma M}{8}\sum_{i=0}^{k}A_{i+1}\|d_i\|^3 \leq \frac{1}{2}\|v_0-x^*\|^2 +\gamma A_0 \left(f(x_0)-f^*\right).
    \end{equation*}
    Since $A_0 = 0$ by initialization, the second term on the right-hand side vanishes. Letting $B_{k+1} = \frac{3\gamma M}{8}\sum_{i=0}^{k}A_{i+1}\|d_i\|^3$, we obtain the desired result \eqref{eq.ms estimate advanced}.
\end{proof}
From \eqref{eq.ms estimate advanced}, it is clear that the oracle complexity of \Cref{alg.ms accelerated utr} reduces to analyzing the growth rate of $A_k$, which is addressed in \citet[Lemma 4.2]{monteiro2013accelerated} and \citet[Lemma 4.3.5]{nesterov2018lectures}. We provide the lower bound of $A_k$ in the following lemma, whose proof is deferred to the \Cref{sec.appendix 2}.
\begin{lemma}
    \label{lem.growth of Ak}
    For any $k\geq 1$, we have
    \begin{equation}
        \label{eq.growth of Ak}
        A_k \geq \left(\frac{\eta}{4}\left (\frac{3\gamma}{4M^2D_0^2}\right)^{1/3}\right)^{3 / 2}\left(\frac{2 k+1}{3}\right)^{7/2} = \Omega\left(\frac{k^{7/2}}{MD_0}\right).
    \end{equation}
\end{lemma}
Hence, as a direct consequence of \eqref{eq.ms estimate advanced} and \eqref{eq.growth of Ak}, the complexity of \Cref{alg.ms accelerated utr} can be summarized as follows.
\begin{theorem}
    [{\citet[Theorem 4.3.2]{nesterov2018lectures}}]\label{thm.outer loop complexity}
    Suppose Assumption~\ref{assm.lipschitz} and Assumption~\ref{assm.solvable}
    hold, and suppose that \eqref{eq.relation lambda sigma} is satisfied at every non-terminated
    iteration. Then it takes \Cref{alg.ms accelerated utr} $O\left(\left(\frac{MD_0^3}{\epsilon_f}\right)^{2/7}\right)$ iterations to find $\epsilon_f$-function-value solutions as in \Cref{def.approximate solution}.
\end{theorem}
\subsubsection{Number of oracle calls of \Cref{alg.bisection ms}}\label{sec.validty of bisection}
Next, we elaborate on how R\&B~(\Cref{alg.bisection ms}) safeguards \eqref{eq.relation lambda sigma}, and then provide the estimate of the number of \ref{eq.newTR} oracles needed during this procedure. We omit the subscripts in the analysis for the bisection for simplicity. We now define the auxiliary function in the analysis, which is bivariate in $y$ and $\sigma$.
\begin{equation}
    \label{eq.linesearch quantity}
    \psi(\sigma,y) := \frac{1}{\sigma}\left \| \left(\nabla^2 f(y)+\sigma I \right)^{-1}\nabla f(y)\right \|, \ \sigma>0, \ y\in\mathbb{R}^n.
\end{equation}
Some basic analyses on local perturbation in $\sigma$ and $y$ are deferred to \Cref{sec.appendix 2}~(see \Cref{lem.psi property wrt sigma} and \Cref{lem.psi property wrt y}).

The below \Cref{lem.reduce line search} means the analysis of auxiliary function $\psi(\sigma,y(\sigma))$ can be simplified by focusing solely on $\sigma$, reducing the search procedure into a one-dimensional problem. The proof of \Cref{lem.reduce line search} can be found in \Cref{sec.app-c}.
\begin{lemma}
    \label{lem.reduce line search}
    At line \ref{line.goal of bisection ms} of \textnormal{R\&B}~(\Cref{alg.bisection ms}), if case i) occurs, i.e.,
    \begin{equation}
        \label{eq.bracket small lambda}
        \lambda = 0, \quad \|d\|<\frac{\eta}{M}\sigma,
    \end{equation}
    then $\psi(\sigma,y(\sigma))<\frac{\eta}{M}$.  Otherwise, if case ii) occurs, i.e.,
    \begin{equation}
        \label{eq.bracket big lambda}
        \lambda > (\theta-1)\sigma,
    \end{equation}
    then $\psi(\sigma,y(\sigma))>\frac{1}{M}$. As a result, if $\sigma$ satisfies
    \begin{equation}
        \label{eq.reduce line search interval}
        \frac{\eta}{M}\leq \psi(\sigma,y(\sigma)) \leq \frac{1}{M},
    \end{equation}
    then \eqref{eq.relation lambda sigma} holds.
\end{lemma}

Now we can validate the choice of $\sigma_-,\sigma_+$ as qualified bracketing points.
\begin{lemma}
    \label{lem.bracketing points}
    Suppose
    \begin{equation}\label{eq.bounded gradient y}
        \|\nabla f(y(\sigma))\|\leq G_0, \quad \forall \sigma>0,
    \end{equation}
    and we let
    \begin{equation}
        \label{eq.bracket points}
        \sigma_- = \sqrt{\frac{2M\epsilon_g}{1+2\theta}}, \quad \sigma_+ = \sqrt{\frac{MG_0}{\eta}}.
    \end{equation}
    Then if \textnormal{R\&B}~(\Cref{alg.bisection ms}) does not early terminate, we have
    \begin{equation*}
        \psi\left (\sigma_-,y(\sigma_-)\right) >\frac{1}{M}, \quad \psi\left (\sigma_+,y(\sigma_+)\right) <\frac{\eta}{M}.
    \end{equation*}
\end{lemma}
\begin{proof}
    First, we show that $\psi(\sigma_+,y(\sigma_+))<\frac{\eta}{M}$,
    \begin{equation*}
        \begin{aligned}
            \psi(\sigma_+,y(\sigma_+)) & = \frac{1}{\sigma_+} \|(\nabla^2 f(y(\sigma_+))+\sigma_+ I)^{-1}\nabla f(y(\sigma_+))\|            \\
                                       & \leq \frac{1}{\sigma_+^2}\|\nabla f(y(\sigma_+))\| \leq \frac{1}{\sigma_+^2} G_0 = \frac{\eta}{M}.
        \end{aligned}
    \end{equation*}
    The second line is from $\nabla^2 f(y(\sigma_+))\succeq 0$, \eqref{eq.bounded gradient y} and \eqref{eq.bracket points}. For the other statement, from \Cref{lem.early termination ms}, we can conclude that if the \Cref{alg.bisection ms} does not early terminate, we have $\lambda_- > (\theta-1)\sigma_-$, and from \Cref{lem.reduce line search} $\psi(\sigma_-,y(\sigma_-))>\frac{1}{M}.$
\end{proof}
Under the bounded gradient assumption \eqref{eq.bounded gradient y}, during the bisection, we have
\begin{gather}
    \label{eq.bisection condition}
    \psi_+:=\psi(\sigma_+,y(\sigma_+))<\frac{\eta}{M}, \quad \psi_-:=\psi(\sigma_-,y(\sigma_-))>\frac{1}{M},\\
    \label{eq.bracket points condition} \sigma_- \geq \sqrt{\frac{2M\epsilon_g}{1+2\theta}}, \quad \sigma_+ \leq \sqrt{\frac{MG_0}{\eta}}.
\end{gather}
These relations ensure that valid bracketing points are maintained throughout the search. Leveraging the property of the curve $\psi(\sigma,y(\sigma))$, we can then derive the complexity of R\&B~(\Cref{alg.bisection ms}) as in \Cref{lem.complexity bisection}. A formal statement and proof of this result are provided in \Cref{sec.appendix 2} (\Cref{coro.bisection bound}).
\begin{lemma}
    \label{lem.complexity bisection}
    Assume that Assumptions~\ref{assm.lipschitz}, Assumption~\ref{assm.solvable} and conditions \eqref{eq.bounded gradient y} hold. Also, assume \begin{equation}\label{eq.bounded assm point main text}
        \|x-x^*\|\leq M_0, \ \|v-x^*\|\leq M_0,
    \end{equation}
    for some $M_0>0$, then the number of \ref{eq.newTR} oracle calls during the bisection is $O\left(\log \left(1/\epsilon_g\right)\right)$.
\end{lemma}
\subsubsection{Oracle complexity of \Cref{alg.ms accelerated utr}}
To derive the final oracle complexity result of \Cref{alg.ms accelerated utr}, one final step remains: we must get rid of the boundedness assumption used in \eqref{eq.bounded gradient y} and \eqref{eq.bounded assm point main text}. We establish this in \Cref{lem.bounded iterate}, the proof of which can be found in \Cref{sec.app-c}.
\begin{lemma}
    \label{lem.bounded iterate}
    Suppose Assumption~\ref{assm.lipschitz} and Assumption~\ref{assm.solvable} hold. For every $k\geq 0$,
    \begin{equation}
        \label{eq.bounded iterate}
        \|x_k- x^*\| \leq \left(\frac{4}{\sqrt{3\gamma}}+1\right)D_0, \quad \|v_k-x^*\|\leq D_0.
    \end{equation}
    As a consequence,
    \begin{gather}
        \label{eq.bounded iterate y}
        \|y_k(\sigma)-x^*\| \leq \left(\frac{4}{\sqrt{3\gamma}}+1\right)D_0, \quad \forall \sigma>0,\\
        \label{eq.ub yk sigma}
        \|\nabla f(y_k(\sigma))\| \leq G_0:= \left(\frac{4}{\sqrt{3\gamma}}+1\right)\|\nabla ^2 f^*\|D_0+\frac{M}{2}\left(\frac{4}{\sqrt{3\gamma}}+1\right)^2 D_0^2.
    \end{gather}
\end{lemma}
By this lemma, we know \eqref{eq.bounded gradient y} holds, and \eqref{eq.bounded assm point main text} holds uniformly for all $k\geq0$ with $M_0 =\left(\frac{4}{\sqrt{3\gamma}}+1\right)D_0 $. Now we finally arrive at the final theorem of \Cref{alg.ms accelerated utr}, as a consequence of \Cref{thm.outer loop complexity}, \Cref{lem.complexity bisection}, and \Cref{lem.bounded iterate}.
\begin{theorem}
    \label{thm.final complexity of msatr}
    Suppose Assumption~\ref{assm.lipschitz} and Assumption~\ref{assm.solvable} hold, it takes at most
    \begin{equation}
        \label{eq.final complexity of msatr}
        O\left(\epsilon_f^{-2/7}\log\left(1/\epsilon_g\right)\right)
    \end{equation}
    \ref{eq.newTR} oracles for \Cref{alg.ms accelerated utr} to find an $(\epsilon_f,\epsilon_g)$-approximate solution as in \Cref{def.approximate solution}.
\end{theorem}
While the proposed algorithm attains a near-optimal global oracle complexity rate, it fails to balance the global guarantees and the local efficiency. The reason behind is that the primal regularizer $\sigma$ must be fixed before the extrapolation point is determined in the extragradient framework, which prevents effective exploitation of local geometric structures and limits faster local convergence.
\section{Numerical Experiments}\label{sec.experiments}
In this section, we present the numerical experiments to validate the global and local behavior of the proposed methods\footnote{Our implementation is available at \url{https://github.com/bzhangcw/UTR.jl/}.}. All experiments are conducted on a single machine with a 14-core Apple M4 Pro CPU and 48GB LPDDR5 RAM.
We conduct experiments on the regularized logistic regression problem, which is defined as follows:
\begin{equation}
    \label{eq.logistic regression}
    f(x) = \frac{1}{N} \sum_{i=1}^N \log (1 + \exp (-b_i a_i^T x)) + \frac{\gamma}{2}\|x\|^2,
\end{equation}
where $a_i \in \mathbb{R}^n$ and $b_i \in \{-1, 1\}$, $\gamma = 10^{-4}$ is the regularization parameter.
As mentioned, this problem is a notorious example where many accelerated SOMs are not as competitive as the classical Newton-type methods \cite{carmon2022optimal,chen2022accelerating,zhang2012on} numerically.
Besides, a well-known estimate for the Lipschitz constant $\widehat M$ of $\nabla^2 f$ can be specified as follows:
$$\widehat M = \left\|\frac{1}{N} \sum_{i=1}^N a_i a_i^{\top}\right\| \max _{i \in[N]}\left\|a_i\right\|.$$
Although the estimate is conservative~\cite{song2021unified}, we adopt it to isolate the basic algorithm frameworks from other practical enhancements, such as the adaptive adjustment of Lipschitz constants, see \cite{mishchenko2023regularized,cartis2011adaptive}.

We implement two different accelerated TR methods (\atr{} and \atrms{}) and compare them to some state-of-the-art SOMs, including
\begin{itemize}[leftmargin=*]
    \item The cubic regularized Newton method (\arc{}, \citet{nesterov2006cubic}) and its accelerated version (\arca{}, \citet{nesterov2008accelerating}).
    \item A non-accelerated TR method using the \ref{eq.newTR} oracle, by setting both $(\sigma_k, r_k)$ proportionally to $\|\nabla f(x_k)\|^{1/2}$ similar to \cite{jiang2026beyond}.
          We test two non-accelerated TR methods $\utr{}~\texttt{(1)}, \utr{}~\texttt{(2)}$, using different Lipschitz estimates $\tfrac{\widehat M}{2}, \widehat M$, respectively. The purpose is to present the sensitivity of Lipschitz constants and set a fair comparison to \arc{}.
\end{itemize}
\begin{figure}[!t]
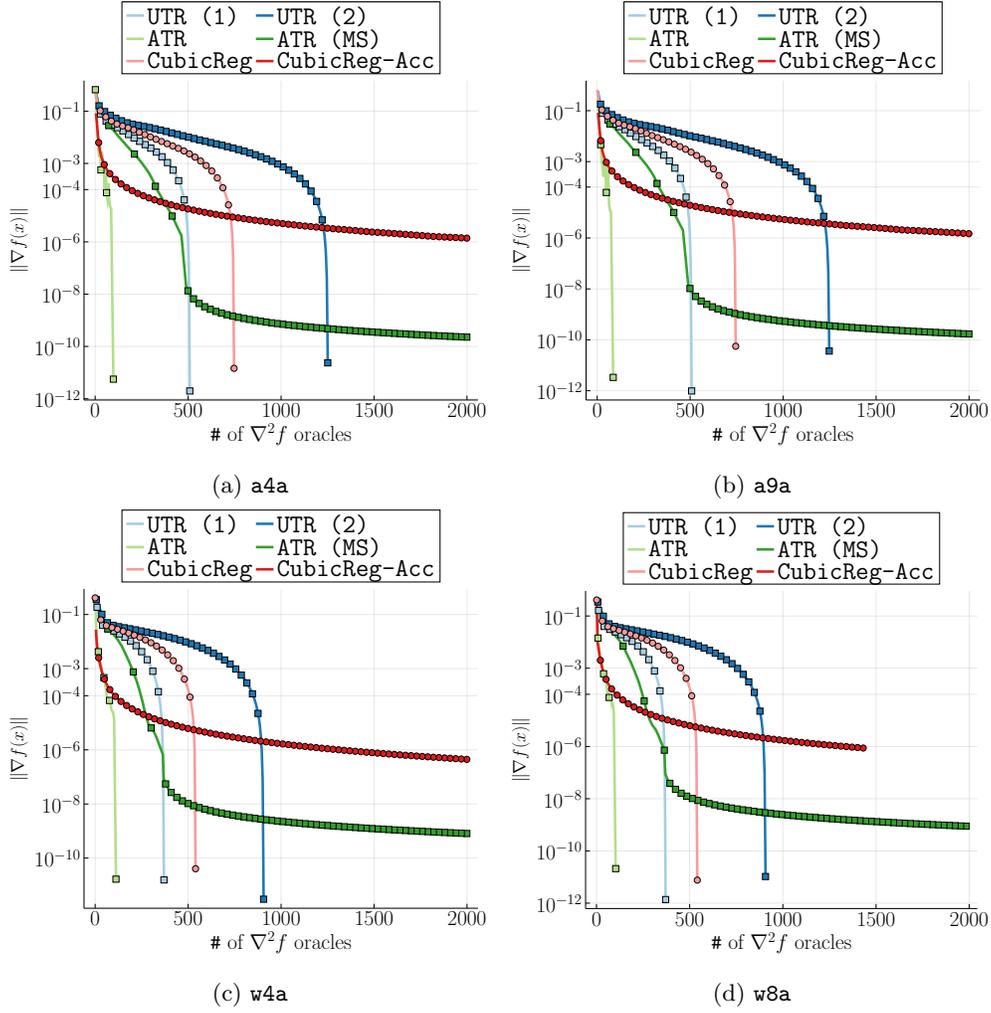

    \small
    \centering
    \subcaptionbox{\texttt{a4a}, $N=4781, n=122$, nnz: 11.36\%.}[0.42\columnwidth]{
        \resizebox{0.42\columnwidth}{!}{\input{e-logistic-a4a-k}}
    }
    \subcaptionbox{\texttt{a9a}, $N=32561, n=123$, nnz: 11.27\%.}[0.42\columnwidth]{
        \resizebox{0.42\columnwidth}{!}{\input{e-logistic-a9a-k}}
    }
    \subcaptionbox{\texttt{w4a}, $N=7366, n=300$, nnz: 3.89\%.}[0.42\columnwidth]{
        \resizebox{0.42\columnwidth}{!}{\input{e-logistic-w4a-k}}
    }
    \subcaptionbox{\texttt{w8a}, $N=49749, n=300$, nnz: 3.88\%.}[0.42\columnwidth]{
        \resizebox{0.42\columnwidth}{!}{\input{e-logistic-w8a-k}}
    }

    \caption{Logistic regression using the LIBSVM datasets.}\label{fig.logistic regression}
    \normalsize
\end{figure}

The subproblems arising in \arc{} and \arca{} are solved by a 1-D line-search strategy according to \cite{nesterov2018lectures}. Similarly, \ref{eq.newTR} (in \utr{}, \atr{} and \atrms{}) is solved by searching the dual variable. All methods use exact Hessian evaluation and Cholesky factorization to solve the linear systems.
Since these SOMs use different subproblems, and theoretically, the complexity rates to solve them vary from $O\left(\log\left(1/\epsilon\right)\right)$ (for subproblems in a cubic regularized method \cite{nesterov2006cubic}) to $O\left(\log\log\left(1/\epsilon\right)\right)$ (for \ref{eq.newTR} \cite{ye1991new,vavasis1990proving}),
we report the number of Hessian evaluations needed in the method. In \Cref{fig.logistic regression}, the performance of the SOMs on some LIBSVM datasets\footnote{For details, see \url{https://www.csie.ntu.edu.tw/cjlin/libsvmtools/datasets/}} is reported; each row of the instances has similar sparsity and dimension but different sample size.

We could have several observations. Firstly, for non-accelerated methods, a trend of local superlinear convergence can be observed in \utr{}~\texttt{(1)}, \utr{}~\texttt{(2)} and \arc{}. We could conclude that these three methods are comparable.
Secondly, in the beginning of the iterations, all accelerated methods, including \atr{}, \atrms{}, and \arca{}, converge faster than the non-accelerated methods (e.g., \utr{}~\texttt{(1)}, \utr{}~\texttt{(2)}, \arc{}). This confirms the effectiveness of the global acceleration.

Thirdly, the global convergence does come at the expense of local superlinear convergence.
Notably, \atr{} is the only accelerated method that has the local rate of superlinear convergence because of the diving track.
Both \atrms{} and \arca{} converge sublinearly in the local regime, in which \atrms{} is slightly better because of its superior $\tilde O(\epsilon_f^{-2/7})$ non-asymptotic performance. These results are in accordance with what is predicted in the convergence analysis.

\newpage
\addcontentsline{toc}{section}{References}

\bibliographystyle{plainnat}
\bibliography{ref_arxiv}

@article{song2021unified,
  title   = {Unified acceleration of high-order algorithms under general {Hölder} continuity},
  author  = {Song, Chaobing and Jiang, Yong and Ma, Yi},
  journal = {SIAM Journal on Optimization},
  year    = {2021},
  volume  = {31},
  number  = {3},
  pages   = {1797--1826}
}

@article{mishchenko2023regularized,
  title   = {Regularized {Newton} {Method} with {Global} $\mathcal{O}(1/k^2)$ {Convergence}},
  author  = {Mishchenko, Konstantin},
  journal = {SIAM Journal on Optimization},
  year    = {2023},
  volume  = {33},
  number  = {3},
  pages   = {1440--1462}
}

@inproceedings{schulman2015trust,
  title     = {Trust region policy optimization},
  author    = {Schulman, John and Levine, Sergey and Abbeel, Pieter and Jordan, Michael and Moritz, Philipp},
  booktitle = {International Conference on Machine Learning (ICML)},
  year      = {2015}
}

@article{chen2022accelerating,
  title   = {Accelerating adaptive cubic regularization of {Newton}'s method via random sampling},
  author  = {Chen, Xi and Jiang, Bo and Lin, Tianyi and Zhang, Shuzhong},
  journal = {Journal of Machine Learning Research},
  year    = {2022},
  volume  = {23},
  number  = {1},
  pages   = {3904--3941}
}

@incollection{ye1991new,
  title     = {A {New} {Complexity} {Result} on {Minimization} of a {Quadratic} {Function} with a {Sphere} {Constraint}},
  author    = {Ye, Yinyu},
  booktitle = {Recent {Advances} in {Global} {Optimization}},
  publisher = {Princeton University Press},
  year      = {1991},
  pages     = {19--31}
}

@article{curtis2018concise,
  title   = {Concise complexity analyses for trust region methods},
  author  = {Curtis, Frank E and Lubberts, Zachary and Robinson, Daniel P},
  journal = {Optimization Letters},
  year    = {2018},
  volume  = {12},
  pages   = {1713--1724}
}

@book{nocedal1999numerical,
  title     = {Numerical {Optimization}},
  author    = {Nocedal, Jorge and Wright, Stephen J},
  publisher = {Springer},
  year      = {1999}
}

@article{doikov2020contracting,
  title   = {Contracting proximal methods for smooth convex optimization},
  author  = {Doikov, Nikita and Nesterov, Yurii},
  journal = {SIAM Journal on Optimization},
  year    = {2020},
  volume  = {30},
  number  = {4},
  pages   = {3146--3169}
}

@article{curtis2017trust,
  title   = {A trust region algorithm with a worst-case iteration complexity of $\mathcal{O}(\epsilon^{-3/2})$ for nonconvex optimization},
  author  = {Curtis, Frank E and Robinson, Daniel P and Samadi, Mohammadreza},
  journal = {Mathematical Programming},
  year    = {2017},
  volume  = {162},
  pages   = {1--32}
}

@inproceedings{hamad2022consistently,
  title     = {A consistently adaptive trust-region method},
  author    = {Hamad, Fadi and Hinder, Oliver},
  booktitle = {Advances in Neural Information Processing Systems (NeurIPS)},
  year      = {2022}
}

@article{curtis2021trust,
  title   = {Trust-region {Newton-CG} with strong second-order complexity guarantees for nonconvex optimization},
  author  = {Curtis, Frank E and Robinson, Daniel P and Royer, Cl{\'e}ment W and Wright, Stephen J},
  journal = {SIAM Journal on Optimization},
  year    = {2021},
  volume  = {31},
  number  = {1},
  pages   = {518--544}
}

@book{nesterov2018lectures,
  title     = {Lectures on {Convex} {Optimization}},
  author    = {Nesterov, Yurii},
  publisher = {Springer},
  year      = {2018}
}

@article{nesterov2006cubic,
  title   = {Cubic regularization of {Newton} method and its global performance},
  author  = {Nesterov, Yurii and Polyak, Boris T},
  journal = {Mathematical Programming},
  year    = {2006},
  volume  = {108},
  number  = {1},
  pages   = {177--205}
}

@book{cartis2022evaluation,
  title     = {Evaluation Complexity of Algorithms for Nonconvex Optimization: Theory, Computation and Perspectives},
  author    = {Cartis, Coralia and Gould, Nicholas IM and Toint, Philippe L},
  publisher = {SIAM},
  year      = {2022}
}

@article{cartis2011adaptive,
  title   = {Adaptive cubic regularisation methods for unconstrained optimization. Part I: motivation, convergence and numerical results},
  author  = {Cartis, Coralia and Gould, Nicholas IM and Toint, Philippe L},
  journal = {Mathematical Programming},
  year    = {2011},
  volume  = {127},
  number  = {2},
  pages   = {245--295}
}

@article{cartis2011adaptive_i,
  title   = {Adaptive cubic regularisation methods for unconstrained optimization. Part II: worst-case function-and derivative-evaluation complexity},
  author  = {Cartis, Coralia and Gould, Nicholas IM and Toint, Philippe L},
  journal = {Mathematical Programming},
  year    = {2011},
  volume  = {130},
  number  = {2},
  pages   = {295--319}
}

@book{conn2000trust,
  title     = {Trust {Region} {Methods}},
  author    = {Conn, Andrew R and Gould, Nicholas IM and Toint, Philippe L},
  publisher = {SIAM},
  year      = {2000}
}

@article{nesterov2008accelerating,
  title   = {Accelerating the cubic regularization of {Newton}’s method on convex problems},
  author  = {Nesterov, Yu},
  journal = {Mathematical Programming},
  year    = {2008},
  volume  = {112},
  number  = {1},
  pages   = {159--181}
}

@article{monteiro2013accelerated,
  title   = {An accelerated hybrid proximal extragradient method for convex optimization and its implications to second-order methods},
  author  = {Monteiro, Renato DC and Svaiter, Benar Fux},
  journal = {SIAM Journal on Optimization},
  year    = {2013},
  volume  = {23},
  number  = {2},
  pages   = {1092--1125}
}

@inproceedings{carmon2022optimal,
  title     = {Optimal and adaptive {Monteiro-Svaiter} acceleration},
  author    = {Carmon, Yair and Hausler, Danielle and Jambulapati, Arun and Jin, Yujia and Sidford, Aaron},
  booktitle = {Advances in Neural Information Processing Systems (NeurIPS)},
  year      = {2022}
}

@article{yuan2015recent,
  title   = {Recent advances in trust region algorithms},
  author  = {Yuan, Ya-Xiang},
  journal = {Mathematical Programming},
  year    = {2015},
  volume  = {151},
  pages   = {249--281}
}

@inproceedings{yuan2000review,
  title     = {A review of trust region algorithms for optimization},
  author    = {Yuan, Ya-Xiang},
  booktitle = {Iciam},
  year      = {2000}
}

@article{byrd2006knitro,
  title   = {{K}nitro: An integrated package for nonlinear optimization},
  author  = {Byrd, Richard H and Nocedal, Jorge and Waltz, Richard A},
  journal = {Large-scale nonlinear optimization},
  year    = {2006},
  pages   = {35--59}
}

@inproceedings{yao2019trust,
  title     = {Trust region based adversarial attack on neural networks},
  author    = {Yao, Zhewei and Gholami, Amir and Xu, Peng and Keutzer, Kurt and Mahoney, Michael W},
  booktitle = {Conference on Computer Vision and Pattern Recognition (CVPR)},
  year      = {2019}
}

@article{ragonneau2024pdfo,
  title   = {{PDFO}: a cross-platform package for Powell’s derivative-free optimization solvers},
  author  = {Ragonneau, Tom M and Zhang, Zaikun},
  journal = {Mathematical Programming Computation},
  year    = {2024},
  pages   = {1--25}
}

@article{jiang2026beyond,
  title   = {Beyond nonconvexity: A universal trust-region method with new analyses},
  author  = {Jiang, Yuntian and He, Chang and Zhang, Chuwen and Ge, Dongdong and Jiang, Bo and Ye, Yinyu},
  journal = {Journal of Scientific Computing},
  year    = {2026},
  volume  = {106},
  number  = {1},
  pages   = {28}
}

@article{hamad2024simple,
  title   = {A simple and practical adaptive trust-region method},
  author  = {Hamad, Fadi and Hinder, Oliver},
  journal = {arXiv:2412.02079},
  year    = {2024}
}

@article{jiang2020unified,
  title   = {A unified adaptive tensor approximation scheme to accelerate composite convex optimization},
  author  = {Jiang, Bo and Lin, Tianyi and Zhang, Shuzhong},
  journal = {SIAM Journal on Optimization},
  year    = {2020},
  volume  = {30},
  number  = {4},
  pages   = {2897--2926}
}

@article{jiang2021optimal,
  title   = {An Optimal High-Order Tensor Method for Convex Optimization},
  author  = {Jiang, Bo and Wang, Haoyue and Zhang, Shuzhong},
  journal = {Mathematics of Operations Research},
  year    = {2021},
  volume  = {46},
  number  = {4},
  pages   = {1390--1412}
}

@article{baes2009estimate,
  title   = {Estimate sequence methods: extensions and approximations},
  author  = {Baes, Michel},
  journal = {Institute for Operations Research, ETH, Z{\"u}rich, Switzerland},
  year    = {2009},
  volume  = {2},
  number  = {1}
}

@inproceedings{nesterov1983method,
  title     = {A method for solving the convex programming problem with convergence rate {$O(1/k^2)$}},
  author    = {Nesterov, Yurii},
  booktitle = {{Dokl.} {Akad.} {Nauk.} {SSSR}},
  year      = {1983}
}

@article{cartis2010on,
  title   = {On the complexity of steepest descent, {Newton}'s and regularized {Newton}'s methods for nonconvex unconstrained optimization problems},
  author  = {Cartis, Coralia and Gould, Nicholas IM and Toint, Ph L},
  journal = {SIAM Journal on Optimization},
  year    = {2010},
  volume  = {20},
  number  = {6},
  pages   = {2833--2852}
}

@inbook{more1983recent,
  title     = {Recent Developments in Algorithms and Software for Trust Region Methods},
  author    = {Mor{\'e}, J. J.},
  booktitle = {Mathematical Programming The State of the Art: Bonn 1982},
  publisher = {Springer Berlin Heidelberg},
  year      = {1983},
  pages     = {258--287}
}

@article{nesterov2021implementable,
  title   = {Implementable tensor methods in unconstrained convex optimization},
  author  = {Nesterov, Yurii},
  journal = {Mathematical Programming},
  year    = {2021},
  volume  = {186},
  pages   = {157--183}
}

@inproceedings{kovalev2022first,
  title     = {The first optimal acceleration of high-order methods in smooth convex optimization},
  author    = {Kovalev, Dmitry and Gasnikov, Alexander},
  booktitle = {Advances in Neural Information Processing Systems (NeurIPS)},
  year      = {2022}
}

@article{agafonov2023inexact,
  title   = {Inexact tensor methods and their application to stochastic convex optimization},
  author  = {Agafonov, Artem and Kamzolov, Dmitry and Dvurechensky, Pavel and Gasnikov, Alexander and Tak{\'a}{\v{c}}, Martin},
  journal = {Optimization Methods and Software},
  year    = {2023},
  pages   = {1--42}
}

@inproceedings{agafonov2024advancing,
  title     = {Advancing the Lower Bounds: an Accelerated, Stochastic, Second-order Method with Optimal Adaptation to Inexactness},
  author    = {Artem Agafonov and Dmitry Kamzolov and Alexander Gasnikov and Ali Kavis and Kimon Antonakopoulos and Volkan Cevher and Martin Tak{\'a}{\v{c}}},
  booktitle = {International Conference on Learning Representations (ICLR)},
  year      = {2024}
}

@article{lin2022explicit,
  title   = {Explicit second-order min-max optimization methods with optimal convergence guarantee},
  author  = {Lin, Tianyi and Mertikopoulos, Panayotis and Jordan, Michael I},
  journal = {arXiv:2210.12860},
  year    = {2022}
}

@inproceedings{jiang2024accelerated,
  title     = {Accelerated quasi-{Newton} proximal extragradient: Faster rate for smooth convex optimization},
  author    = {Jiang, Ruichen and Mokhtari, Aryan},
  booktitle = {Advances in Neural Information Processing Systems (NeurIPS)},
  year      = {2024}
}

@article{huang2024inexact,
  title   = {Inexact and Implementable Accelerated {Newton} Proximal Extragradient Method for Convex Optimization},
  author  = {Huang, Ziyu and Jiang, Bo and Jiang, Yuntian},
  journal = {arXiv:2402.11951},
  year    = {2024}
}

@article{adachi2019eigenvalue,
  title   = {Eigenvalue-based algorithm and analysis for nonconvex {QCQP} with one constraint},
  author  = {Adachi, Satoru and Nakatsukasa, Yuji},
  journal = {Mathematical Programming},
  year    = {2019},
  volume  = {173},
  number  = {1-2},
  pages   = {79--116}
}

@article{adachi2017solving,
  title   = {Solving the trust-region subproblem by a generalized eigenvalue problem},
  author  = {Adachi, Satoru and Iwata, Satoru and Nakatsukasa, Yuji and Takeda, Akiko},
  journal = {SIAM Journal on Optimization},
  year    = {2017},
  volume  = {27},
  number  = {1},
  pages   = {269--291}
}

@article{rojas2001new,
  title   = {A new matrix-free algorithm for the large-scale trust-region subproblem},
  author  = {Rojas, Marielba and Santos, Sandra A. and Sorensen, Danny C.},
  journal = {SIAM Journal on Optimization},
  year    = {2001},
  volume  = {11},
  number  = {3},
  pages   = {611--646}
}

@article{he2025homogeneous,
  title   = {Homogeneous second-order descent framework: a fast alternative to {Newton-type} methods},
  author  = {He, Chang and Jiang, Yuntian and Zhang, Chuwen and Ge, Dongdong and Jiang, Bo and Ye, Yinyu},
  journal = {Mathematical Programming},
  year    = {2025}
}

@article{zhang2025homogeneous,
  title   = {A homogeneous second-order descent method for nonconvex optimization},
  author  = {Zhang, Chuwen and He, Chang and Jiang, Yuntian and Xue, Chenyu and Jiang, Bo and Ge, Dongdong and Ye, Yinyu},
  journal = {Mathematics of Operations Research},
  year    = {2025}
}

@misc{ye2005second,
  title  = {Second {Order} {Optimization} {Algorithms} {I}},
  author = {Ye, Yinyu},
  year   = {2005}
}

@techreport{vavasis1990proving,
  title       = {Proving polynomial-time for sphere-constrained quadratic programming},
  author      = {Vavasis, Stephen A. and Zippel, Richard},
  institution = {Cornell University},
  year        = {1990}
}

@article{jiang2022holderian,
  title   = {Hölderian Error Bounds and {Kurdyka-Łojasiewicz} Inequality for the Trust Region Subproblem},
  author  = {Jiang, Rujun and Li, Xudong},
  journal = {Mathematics of Operations Research},
  year    = {2022},
  volume  = {47},
  number  = {4},
  pages   = {3025--3050}
}

@techreport{more1982newton,
  title       = {{Newton}'s method},
  author      = {Mor{\'e}, Jorge J and Sorensen, Danny C},
  institution = {Argonne National Lab.(ANL), Argonne, IL (United States)},
  year        = {1982}
}

@article{polyak2007newton,
  title   = {{Newton}’s method and its use in optimization},
  author  = {Polyak, Boris T},
  journal = {European Journal of Operational Research},
  year    = {2007},
  volume  = {181},
  number  = {3},
  pages   = {1086--1096}
}

@article{jiang2022generalized,
  title   = {Generalized optimistic methods for convex-concave saddle point problems},
  author  = {Jiang, Ruichen and Mokhtari, Aryan},
  journal = {arXiv:2202.09674},
  year    = {2022}
}

@article{marquesalves2022variants,
  title   = {Variants of the A-{HPE} and large-step A-{HPE} algorithms for strongly convex problems with applications to accelerated high-order tensor methods},
  author  = {Marques Alves, M},
  journal = {Optimization Methods and Software},
  year    = {2022},
  volume  = {37},
  number  = {6},
  pages   = {2021--2051}
}

@article{huang2025approximation,
  title   = {An approximation-based regularized extra-gradient method for monotone variational inequalities},
  author  = {Huang, Kevin and Zhang, Shuzhong},
  journal = {SIAM Journal on Optimization},
  year    = {2025},
  volume  = {35},
  number  = {3},
  pages   = {1469--1497}
}

@article{lin2025perseus,
  title   = {Perseus: A simple and optimal high-order method for variational inequalities},
  author  = {Lin, Tianyi and Jordan, Michael I},
  journal = {Mathematical Programming},
  year    = {2025},
  volume  = {209},
  number  = {1},
  pages   = {609--650}
}

@article{ostroukhov2020tensor,
  title   = {Tensor methods for strongly convex strongly concave saddle point problems and strongly monotone variational inequalities},
  author  = {Ostroukhov, Petr and Kamalov, Rinat and Dvurechensky, Pavel and Gasnikov, Alexander},
  journal = {arXiv:2012.15595},
  year    = {2020}
}

@article{grapiglia2015on,
  title   = {On the convergence and worst-case complexity of trust-region and regularization methods for unconstrained optimization},
  author  = {Grapiglia, Geovani N and Yuan, Jinyun and Yuan, Ya-xiang},
  journal = {Mathematical Programming},
  year    = {2015},
  volume  = {152},
  number  = {1},
  pages   = {491--520}
}

@book{dennisjr1996numerical,
  title     = {Numerical {M}ethods for {U}nconstrained {O}ptimization and {N}onlinear {E}quations},
  author    = {Dennis Jr, John E and Schnabel, Robert B},
  publisher = {SIAM},
  year      = {1996}
}

@article{honguyen2017second,
  title   = {A second-order cone based approach for solving the trust-region subproblem and its variants},
  author  = {Ho-Nguyen, Nam and Kilinc-Karzan, Fatma},
  journal = {SIAM Journal on Optimization},
  year    = {2017},
  volume  = {27},
  number  = {3},
  pages   = {1485--1512}
}

@article{wang2022generalized,
  title   = {The generalized trust region subproblem: solution complexity and convex hull results},
  author  = {Wang, Alex L and K{\i}l{\i}n{\c{c}}-Karzan, Fatma},
  journal = {Mathematical Programming},
  year    = {2022},
  volume  = {191},
  number  = {2},
  pages   = {445--486}
}

@article{xu2017accelerated,
  title   = {Accelerated first-order primal-dual proximal methods for linearly constrained composite convex programming},
  author  = {Xu, Yangyang},
  journal = {SIAM Journal on Optimization},
  year    = {2017},
  volume  = {27},
  number  = {3},
  pages   = {1459--1484}
}

@article{xu2018accelerated,
  title   = {Accelerated primal--dual proximal block coordinate updating methods for constrained convex optimization},
  author  = {Xu, Yangyang and Zhang, Shuzhong},
  journal = {Computational Optimization and Applications},
  year    = {2018},
  volume  = {70},
  number  = {1},
  pages   = {91--128}
}

@article{ghadimi2016accelerated,
  title   = {Accelerated gradient methods for nonconvex nonlinear and stochastic programming},
  author  = {Ghadimi, Saeed and Lan, Guanghui},
  journal = {Mathematical Programming},
  year    = {2016},
  volume  = {156},
  number  = {1},
  pages   = {59--99}
}

@article{lan2012optimal,
  title   = {An optimal method for stochastic composite optimization},
  author  = {Lan, Guanghui},
  journal = {Mathematical Programming},
  year    = {2012},
  volume  = {133},
  number  = {1},
  pages   = {365--397}
}

@article{so2007on,
  title   = {On approximating complex quadratic optimization problems via semidefinite programming relaxations},
  author  = {So, Anthony Man-Cho and Zhang, Jiawei and Ye, Yinyu},
  journal = {Mathematical Programming},
  year    = {2007},
  volume  = {110},
  number  = {1},
  pages   = {93--110}
}

@article{so2011deterministic,
  title   = {Deterministic approximation algorithms for sphere constrained homogeneous polynomial optimization problems},
  author  = {So, Anthony Man-Cho},
  journal = {Mathematical Programming},
  year    = {2011},
  volume  = {129},
  number  = {2},
  pages   = {357--382}
}

@article{curtis2023worst_i,
  title   = {Worst-case complexity of {TRACE} with inexact subproblem solutions for nonconvex smooth optimization},
  author  = {Curtis, Frank E and Wang, Qi},
  journal = {SIAM Journal on Optimization},
  year    = {2023},
  volume  = {33},
  number  = {3},
  pages   = {2191--2221}
}

@article{ouyang2025trust,
  title   = {A trust region-type normal map-based semismooth {Newton} method for nonsmooth nonconvex composite optimization},
  author  = {Ouyang, Wenqing and Milzarek, Andre},
  journal = {Mathematical Programming},
  year    = {2025},
  volume  = {212},
  number  = {1},
  pages   = {389--435}
}

@article{zhang2012on,
  title   = {On the local convergence of a derivative-free algorithm for least-squares minimization},
  author  = {Zhang, Hongchao and Conn, Andrew R},
  journal = {Computational Optimization and Applications},
  year    = {2012},
  volume  = {51},
  number  = {2},
  pages   = {481--507}
}

@article{he2023newton,
  title   = {A {Newton}-CG based augmented Lagrangian method for finding a second-order stationary point of nonconvex equality constrained optimization with complexity guarantees},
  author  = {He, Chuan and Lu, Zhaosong and Pong, Ting Kei},
  journal = {SIAM Journal on Optimization},
  year    = {2023},
  volume  = {33},
  number  = {3},
  pages   = {1734--1766}
}

@article{ge2022cardinal,
  title   = {Cardinal Optimizer ({COPT}) user guide},
  author  = {Ge, Dongdong and Huangfu, Qi and Wang, Zizhuo and Wu, Jian and Ye, Yinyu},
  journal = {arXiv:2208.14314},
  year    = {2022}
}
\newpage
\appendix
\section{Technical proofs in~\Cref{sec.priliminary}}\label{sec.proof pre}
\subsection{Auxiliary results of estimating sequence}
\begin{lemma}
    \label{lem.magnitude a and A}
    For $k\geq 0$, we have
    \begin{equation}
        \label{eq.magnitude a and A}
        \frac{A_{k+1}}{a_k^{3/2}} \geq \frac{\sqrt{2}}{3}.
    \end{equation}
\end{lemma}
\begin{proof}The conclusion follows from that
    $$\frac{A_{k+1}}{a_k^{3/2}}= \frac{(k+1)(k+2)(k+3)}{6}\cdot \frac{2^{3/2}}{(k+1)^{3/2}(k+2)^{3/2}} = \frac{\sqrt{2}(k+3)}{3(k+1)^{1/2}(k+2)^{1/2}} \geq \frac{\sqrt{2}}{3}.$$
\end{proof}
We call a differentiable function $d(x)$ on $\mathbb{R}^n$ uniformly convex \cite[Section 4.2.2]{nesterov2018lectures} of degree $p\geq 2$ with constant $q>0$ if
\begin{equation}
    \label{eq.uniformly convex of p}
    d(y)\geq d(x)+\langle \nabla d(x),y-x\rangle +\frac{q}{p}\|y-x\|^p, \ \forall x,y\in\mathbb{R}^n.
\end{equation}
\begin{lemma}
    \label{lem.property of phi}
    For the function sequence $\phi_k(x)$, they have the following properties
    \begin{enumerate}[label=(\roman*)]
        \item $\phi_k(x)$ is uniformly convex of degree $3$ with constant $12M$.
        \item $v_k$ is the unique minimizer of $\phi_k(x)$, and
              \begin{equation}
                  \label{eq.phi_k phi_low}
                  \phi_k(x)\geq \phi_k^* + 4M\|x-v_k\|^3,
              \end{equation}
              where $\phi_k^* =\min \phi_k(x).$
    \end{enumerate}
\end{lemma}
\begin{proof}
    The first claim is from the definition of $\phi_0$ in \eqref{eq.update phi}, for analysis of uniform convex functions, please refer to \cite[Section 4.2.2]{nesterov2018lectures}. For the second claim, we have
    \begin{equation}
        \label{eq.phi recursive}
        \phi_k(x) = \phi_0(x)+\sum_{i=0}^{k-1}a_i\left (f(x_{i+1})+\langle \nabla f(x_{i+1}),x-x_{i+1}\rangle \right).
    \end{equation}
    From the optimality condition, we have that
    \begin{equation*}
        \nabla \phi_k(x) = \nabla \phi_0(x)+ \sum_{i=0}^{k-1} a_i\nabla f(x_{i+1}) = 24M\|x-x_0\|(x-x_0)+\sum_{i=0}^{k-1} a_i\nabla f(x_{i+1}) =0.
    \end{equation*}
    Solving the optimality condition and noticing the way we update $s_k$ gives
    \begin{equation*}
        x = v_0 - \sqrt{\frac{1}{24M\|s_{k}\|}}s_{k},
    \end{equation*}
    which is $v_k$. Then \eqref{eq.phi_k phi_low} holds from the fact that $\phi_k(x)$ is uniformly convex of degree 3.
\end{proof}
\subsection*{Proof to~\Cref{lem.bound next gradient}}
\begin{proof}
    From \eqref{eq.second-order exp}, we have
    \begin{equation*}
        \|\nabla f(x+d)- \nabla f(x) -\nabla ^2 f(x) d\| \leq \frac{M}{2}\|d\|^2.
    \end{equation*}
    By \eqref{eq.second-order update} and the above, we have
    \begin{equation}\label{eq.bounded by squared step}
        \|\nabla f(x+d)+\mu d\| \leq \frac{M}{2}\|d\|^2.
    \end{equation}
    Applying the triangle inequality and rearranging items, we can derive \eqref{eq.bound next gradient}.
\end{proof}
\subsection*{Proof to~\Cref{lem.inner product g and d}}
\begin{proof}
    Squaring both sides of \eqref{eq.bounded by squared step} and rearranging items, we have
    \begin{equation*}
        \begin{aligned}
            2\mu \langle \nabla f(x+d),-d \rangle & \geq \|\nabla f(x+d)\|^2 +\mu^2 \|d\|^2 -\frac{M^2}{4}\|d\|^4 \\
                                                  & \geq \|\nabla f(x+d)\|^2 +\frac{3}{4}\mu^2 \|d\|^2.
        \end{aligned}
    \end{equation*}
    The second line is due to $\mu\geq M\|d\|$, dividing both sides by $2\mu$ gives \eqref{eq.inner product g d}. Further, when $\mu \leq 2M\|d\|$, the above gives
    \begin{equation*}
        4M\|d\| \langle \nabla f(x+d),-d \rangle \geq \|\nabla f(x+d)\|^2 +\frac{3}{4} M^2 \|d\|^4,
    \end{equation*}
    which yields
    \begin{equation*}
        \langle \nabla f(x+d),-d \rangle \geq \frac{\|\nabla f(x+d)\|^2}{4M\|d\|}+\frac{3}{16}M\|d\|^3.
    \end{equation*}
    Consider an auxiliary function $h(t) = \frac{\|\nabla f(x+d)\|^2}{4Mt}+\frac{3}{16}Mt^3$ where $t\geq 0$, taking derivatives gives $h(t)$ achieves its minimum at $t^*= \frac{\sqrt{2}\|\nabla f(x+d)\|^{1/2}}{\sqrt{3M}}$, plugging $t^*$ back gives \eqref{eq.inner product g d advanced}.
\end{proof}
\section{Technical proofs in~\Cref{sec.first variant}}\label{sec.proof variant 1}
\subsection*{Proof to \Cref{lem.bound target interval}}
To prove \Cref{lem.bound target interval}, we first introduce the following lemma, which discusses the property of the auxiliary function $g_y(\mu)$.
\begin{lemma}
    \label{lem.property g}
    For any $y\in \mathbb{R}^n$, $g_y(\mu)$ is continuously differentiable and monotonically increasing for $\mu \in \left(0,+\infty\right)$ and
    \begin{equation}
        \label{eq.derivative}
        g_y'(\mu) = \frac{1}{\|\left (\nabla ^2 f(y)+\mu I \right)^{-1}\nabla f(y)\|}+\frac{\mu}{\| (\nabla ^2 f(y)+\mu I )^{-1}\nabla f(y)\|^3}\sum_{i=1}^n\frac{\beta_i^2}{ (\zeta_i+\mu )^3},
    \end{equation}
    and it is bounded both below and above
    \begin{equation}
        \label{eq.derivative bound}
        0< g_y'(\mu) \leq \frac{2}{\| (\nabla ^2 f(y)+\mu I )^{-1}\nabla f(y)\|}.
    \end{equation}
    Where $\zeta_i$ is the $i$-th eigenvalue of $\nabla ^2 f(y)$, $\beta_i = \nabla f(y)^T v_i$, $v_i$ denotes the eigenvector corresponding to $\zeta_i$.
\end{lemma}
\begin{proof}
    Since $\nabla^2 f(y)+\mu I \succ 0$ for $\mu>0$, we can apply eigen decomposition to it, then
    \begin{equation*}
        \left \| (\nabla ^2 f(y)+\mu I )^{-1}\nabla f(y)\right\|=\sqrt{\sum_{i=1}^n \frac{\beta_i^2}{\left (\zeta_i+\mu\right )^2}}, \ g_y(\mu) = \frac{\mu}{\sqrt{\sum_{i=1}^n \frac{\beta_i^2}{\left (\zeta_i+\mu\right )^2}}}.
    \end{equation*}
    Obviously $g_y(\mu)$ is increasing in $\mu \in \left (0,+\infty\right ]$. Also, through some basic calculations, we have
    \begin{equation*}
        \begin{aligned}
            g_y'(\mu) & = \frac{1}{\sqrt{\sum_{i=1}^n \frac{\beta_i^2}{\left (\zeta_i+\mu\right )^2}}}+\mu \frac{1}{\left ( \sum_{i=1}^n \frac{\beta_i^2}{\left (\zeta_i+\mu\right )^2}\right )^{3/2}}\sum_{i=1}^n\frac{\beta_i^2}{\left (\zeta_i+\mu\right )^3} \\
                      & = \frac{1}{\| (\nabla ^2 f(y)+\mu I )^{-1}\nabla f(y)\|}+\frac{\mu}{\| (\nabla ^2 f(y)+\mu I )^{-1}\nabla f(y)\|^3}\sum_{i=1}^n\frac{\beta_i^2}{\left (\zeta_i+\mu\right )^3}                                                            \\
                      & \leq_{(a)} \frac{1}{ \| (\nabla ^2 f(y)+\mu I )^{-1}\nabla f(y)\|}+\frac{\mu}{\| (\nabla ^2 f(y)+\mu I )^{-1}\nabla f(y)\|^3}\sum_{i=1}^n\frac{\beta_i^2}{\mu\left (\zeta_i+\mu\right )^2}                                               \\
                      & =  \frac{2}{\| (\nabla ^2 f(y)+\mu I )^{-1}\nabla f(y)\|}.
        \end{aligned}
    \end{equation*}
    $(a)$ is because of $\zeta_i \geq 0$ for $i=1,\ldots,n$.
\end{proof}

Now we are ready to formally prove \Cref{lem.bound target interval}.
\begin{proof}
    From the continuity and monotonicity of $g_y(\mu)$, we know there exists interval $[\mu_l,\mu_u]$ with \eqref{eq.target interval bracketing points} and $M\leq g_y(\mu)\leq 2M$ for all $\mu$ in this interval~(intermediate value theorem).

    Now we show that the length of the target interval for $\mu$ is bounded below; By the mean value theorem, there exists $\xi\in \left [\mu_l,\mu_u\right]$, such that
    \begin{equation*}
        M=g_y(\mu_u)-g_y(\mu_l) = g_y'(\xi)(\mu_u-\mu_l),
    \end{equation*}
    combine the above with \eqref{eq.derivative bound} and \eqref{eq.target interval bracketing points} we have
    \begin{equation*}
        \mu_u-\mu_l =  \frac{1}{g_y'(\xi)}\left(g_y(\mu_u)-g_y(\mu_l)\right) \geq \frac{M\| (\nabla ^2 f(y)+\xi I  )^{-1}\nabla f(y)\|}{2} \geq_{(a)} \frac{M\| (\nabla ^2 f(y)+\mu_u I  )^{-1}\nabla f(y)\|}{2}.
    \end{equation*}
    $(a)$ is because of $\mu_u >\xi$.
\end{proof}
\subsection*{Proof to \Cref{lem.bound initial interval}}
\begin{proof}
    Note
    \[
        r_y(\mu)
        =
        \|(\nabla^2 f(y)+\mu I)^{-1}\nabla f(y)\|.
    \]
    Let the target interval for \(r\) be \([r_l,r_u]\), where
    \[
        r_l=r_y(\mu_u),
        \qquad
        r_u=r_y(\mu_l).
    \]
    Since \(r_y(\mu)\) is decreasing in \(\mu\), we have \(r_l<r_u\). The
    bisection procedure terminates whenever the current radius \(r\) lies in
    \([r_l,r_u]\). Hence the number of bisection steps is bounded by
    \begin{equation}
        \label{eq.rb-complexity-basic}
        O\left(
        \log\frac{r_+-r_-}{r_u-r_l}
        \right)
        \le
        O\left(
        \log\frac{r_+}{r_u-r_l}
        \right).
    \end{equation}

    We first lower bound the length \(r_u-r_l\). Let
    \[
        \nabla^2 f(y)=V\operatorname{diag}(\zeta_1,\ldots,\zeta_n)V^\top,
        \qquad
        \beta_i=\nabla f(y)^\top v_i.
    \]
    By convexity and Assumption~\ref{assm.bounded hessian}, we have
    \[
        0\le \zeta_i\le \kappa_H,
        \qquad i=1,\ldots,n.
    \]
    For \(\mu>0\),
    \[
        r_y(\mu)
        =
        \left(
        \sum_{i=1}^n
        \frac{\beta_i^2}{(\zeta_i+\mu)^2}
        \right)^{1/2}.
    \]
    Therefore,
    \begin{equation}
        \label{eq.derivative-ry}
        r_y'(\mu)
        =
        -
        \frac{
            \sum_{i=1}^n
            \frac{\beta_i^2}{(\zeta_i+\mu)^3}
        }{
            r_y(\mu)
        }.
    \end{equation}
    By the mean value theorem, there exists
    \(\xi\in[\mu_l,\mu_u]\) such that
    \begin{equation}
        r_u-r_l
        =
        r_y(\mu_l)-r_y(\mu_u)
        =
        -r_y'(\xi)(\mu_u-\mu_l)
        =
        \frac{
            \sum_{i=1}^n
            \frac{\beta_i^2}{(\zeta_i+\xi)^3}
        }{
            r_y(\xi)
        }
        (\mu_u-\mu_l).
        \label{eq.radius-interval-mvt}
    \end{equation}
    Since \(\xi\le \mu_u\) and \(\zeta_i\le \kappa_H\), we have
    \[
        \frac{1}{\zeta_i+\xi}
        \ge
        \frac{1}{\kappa_H+\mu_u}.
    \]
    Thus
    \begin{equation}
        \sum_{i=1}^n
        \frac{\beta_i^2}{(\zeta_i+\xi)^3}
        =
        \sum_{i=1}^n
        \frac{\beta_i^2}{(\zeta_i+\xi)^2}
        \frac{1}{\zeta_i+\xi}
        \ge
        \frac{1}{\kappa_H+\mu_u}
        \sum_{i=1}^n
        \frac{\beta_i^2}{(\zeta_i+\xi)^2}
        =
        \frac{r_y(\xi)^2}{\kappa_H+\mu_u}.
        \label{eq.s3-lower-by-s2}
    \end{equation}
    Combining \eqref{eq.radius-interval-mvt} and \eqref{eq.s3-lower-by-s2}, we obtain
    \begin{equation}
        \label{eq.radius-interval-lower-1}
        r_u-r_l
        \ge
        \frac{r_y(\xi)}{\kappa_H+\mu_u}
        (\mu_u-\mu_l).
    \end{equation}
    Since \(r_y(\mu)\) is decreasing and \(\xi\le \mu_u\), we have
    \[
        r_y(\xi)\ge r_y(\mu_u).
    \]
    Moreover, by \Cref{lem.bound target interval},
    \[
        \mu_u-\mu_l
        \ge
        \frac{M r_y(\mu_u)}{2}.
    \]
    Therefore,
    \begin{equation}
        \label{eq.radius-interval-lower-2}
        r_u-r_l
        \ge
        \frac{M r_y(\mu_u)^2}{2(\kappa_H+\mu_u)}.
    \end{equation}
    Finally, since \(0\preceq \nabla^2 f(y)\preceq \kappa_H I\), we have
    \[
        r_y(\mu_u)
        =
        \|(\nabla^2 f(y)+\mu_u I)^{-1}\nabla f(y)\|
        \ge
        \frac{\|\nabla f(y)\|}{\kappa_H+\mu_u}.
    \]
    Substituting this into \eqref{eq.radius-interval-lower-2} gives
    \begin{equation}
        \label{eq.radius-interval-lower-final}
        r_u-r_l
        \ge
        \frac{M\|\nabla f(y)\|^2}{2(\kappa_H+\mu_u)^3}.
    \end{equation}

    We now return to the bisection complexity. From \eqref{eq.rb-complexity-basic}
    and \eqref{eq.radius-interval-lower-final}, we obtain
    \begin{equation}
        \label{eq.rb-complexity-1}
        O\left(
        \log\frac{r_+}{r_u-r_l}
        \right)
        \le
        O\left(
        \log
        \frac{
                2r_+(\kappa_H+\mu_u)^3
            }{
                M\|\nabla f(y)\|^2
            }
        \right).
    \end{equation}
    By the definition of \(r_+\) in line~\ref{line.left bracketing point} of
    \Cref{alg.local detection},
    \[
        r_+
        =
        \|(\nabla^2 f(y)+M\|d_+\|I)^{-1}\nabla f(y)\|.
    \]
    Since \(\nabla^2 f(y)\succeq0\), we have
    \[
        r_+
        \le
        \frac{\|\nabla f(y)\|}{M\|d_+\|}.
    \]
    Together with \(\mu_u\le\mu_+\), this yields
    \begin{equation}
        \label{eq.rb-complexity-2}
        O\left(
        \log
        \frac{
                2r_+(\kappa_H+\mu_u)^3
            }{
                M\|\nabla f(y)\|^2
            }
        \right)
        \le
        O\left(
        \log
        \frac{
                2(\kappa_H+\mu_+)^3
            }{
                M^2\|d_+\|\|\nabla f(y)\|
            }
        \right).
    \end{equation}

    When the bisection procedure is invoked, line~\ref{line.check right bracket point}
    of \Cref{alg.local detection} implies
    \(
    \frac{\mu_+}{\|d_+\|}>2M.
    \)
    Since
    \(
    \mu_+
    =
    \frac{\sqrt{2M}}{2}\|\nabla f(y)\|^{1/2},
    \)
    we have
    \(
    \|\nabla f(y)\|
    =
    \frac{2\mu_+^2}{M}
    >
    8M\|d_+\|^2.
    \)
    Therefore,
    \[
        M^2\|d_+\|\|\nabla f(y)\|
        >
        8M^3\|d_+\|^3.
    \]
    Hence \eqref{eq.rb-complexity-2} further implies
    \begin{equation}
        \label{eq.rb-complexity-3}
        O\left(
        \log
        \frac{
                2(\kappa_H+\mu_+)^3
            }{
                M^2\|d_+\|\|\nabla f(y)\|
            }
        \right)
        \le
        O\left(
        \log
        \frac{
                (\kappa_H+\mu_+)^3
            }{
                M^3\|d_+\|^3
            }
        \right).
    \end{equation}

    It remains to use the fact that the early-termination test in
    \Cref{alg.local detection} has not been satisfied before entering R\&B.
    Thus
    \[
        \|d_+\|
        >
        \min\left\{
        \frac{\epsilon_g}{2\kappa_H},
        \sqrt{\frac{\epsilon_g}{M}}
        \right\}.
    \]
    Consequently,
    \[
        \frac{1}{\|d_+\|^3}
        \le
        \max\left\{
        \left(\frac{2\kappa_H}{\epsilon_g}\right)^3,
        \left(\frac{M}{\epsilon_g}\right)^{3/2}
        \right\}.
    \]
    Substituting this into \eqref{eq.rb-complexity-3}, we obtain
    \begin{align*}
        O\left(
        \log
        \frac{
                (\kappa_H+\mu_+)^3
            }{
                M^3\|d_+\|^3
            }
        \right)
         & \le
        O\left(
        \log
        \left[
            \frac{(\kappa_H+\mu_+)^3}{M^3}
            \max\left\{
            \left(\frac{2\kappa_H}{\epsilon_g}\right)^3,
            \left(\frac{M}{\epsilon_g}\right)^{3/2}
            \right\}
            \right]
        \right) \\
         & =
        O\left(
        \log\frac{\kappa_H+\mu_+}{\epsilon_g}
        \right).
    \end{align*}
\end{proof}
\subsection*{Proof to \Cref{lem.enter region}}
\begin{proof}
    Let $R=\|x_0-x^*\|$ and define
    \[
        \rho:=\min\left\{\frac{2\nu}{3M},\frac{2\nu^2}{9\kappa_HM}\right\}.
    \]
    It suffices to show $\|y_k-x^*\|<\rho$, since Assumption~\ref{assm.bounded hessian} and
    $\nabla f(x^*)=0$ imply
    \[
        \|\nabla f(y_k)\|
        =\|\nabla f(y_k)-\nabla f(x^*)\|
        \leq \kappa_H\|y_k-x^*\|.
    \]

    We first derive two elementary bounds. From the estimating-sequence relation with
    $x=x^*$, we have, for every $k\geq1$,
    \[
        f(x_k)-f^*
        \leq \frac{\phi_0(x^*)}{A_k}
        = \frac{48MR^3}{k(k+1)(k+2)}
        \leq \frac{48MR^3}{k^3}.
    \]
    Moreover, using \eqref{eq.phi_k phi_low} at $x=x^*$ and the estimating-sequence
    upper bound,
    \[
        4M\|v_k-x^*\|^3
        \leq \phi_k(x^*)-\phi_k^*
        \leq_{(a)} \phi_0(x^*)
        =8MR^3.
    \]
    $(a)$ is from $\phi_k(x^*)\le A_k f(x^*)+\phi_0(x^*) \le \phi_k^*+\phi_0(x^*)$. Therefore,
    \(
    \|v_k-x^*\|\leq 2^{1/3}R.
    \)

    Next we show that $x_k$ is already in $\mathcal{LQ}_p$ for the values of $k$
    considered below. For any $x$ satisfying
    $\|x-x^*\|=2\nu/(3M)$, Taylor's theorem and \eqref{eq.psd region} give
    \[
        f(x)-f^*
        \geq \frac{\nu}{6}\|x-x^*\|^2
        = \frac{2\nu^3}{27M^2}.
    \]
    By convexity, the same lower bound holds for every
    $x$ with $\|x-x^*\|\geq 2\nu/(3M)$. Hence, if
    \[
        f(x_k)-f^*<\frac{2\nu^3}{27M^2},
    \]
    then $x_k\in\mathcal{LQ}_p$. In particular, the preceding global bound shows that
    this holds whenever
    \[
        k\geq \left(648\right)^{1/3}\frac{MR}{\nu}.
    \]

    For such $k$, \eqref{eq.psd xk} holds along the segment joining $x_k$ and $x^*$,
    and hence
    \[
        \frac{\nu}{6}\|x_k-x^*\|^2
        \leq f(x_k)-f^*
        \leq \frac{48MR^3}{k^3}.
    \]
    Thus
    \[
        \|x_k-x^*\|
        \leq \sqrt{\frac{288MR^3}{\nu k^3}}.
    \]

    Using the definition of $y_k$:
    \(
    y_k=\frac{k}{k+3}x_k+\frac{3}{k+3}v_k,
    \)
    we obtain
    \[
        \|y_k-x^*\|
        \leq \frac{k}{k+3}\|x_k-x^*\|
        +\frac{3}{k+3}\|v_k-x^*\|
        \leq \sqrt{\frac{288MR^3}{\nu k^3}}
        +\frac{3\cdot 2^{1/3}R}{k}.
    \]

    Finally, since Assumption~\ref{assm.bounded hessian} and
    Assumption~\ref{assm.local strongly convex} imply $\kappa_H\geq \nu$, we have
    \[
        \rho
        =\min\left\{\frac{2\nu}{3M},\frac{2\nu^2}{9\kappa_HM}\right\}
        =\frac{2\nu^2}{9\kappa_HM}.
    \]
    If
    \[
        k\geq 36\,\frac{\kappa_H M R}{\nu^2},
    \]
    then this condition also implies
    \[
        k\geq (648)^{1/3}\frac{MR}{\nu},
    \]
    and the two terms above satisfy
    \[
        \sqrt{\frac{288MR^3}{\nu k^3}}
        \leq \frac{\rho}{2},
        \qquad
        \frac{3\cdot 2^{1/3}R}{k}
        \leq \frac{\rho}{2}.
    \]
    Consequently,
    \(
    \|y_k-x^*\|\leq \rho.
    \)
    Therefore,
    \[
        \|y_k-x^*\|\leq \frac{2\nu}{3M},
        \qquad
        \|\nabla f(y_k)\|
        \leq \kappa_H\|y_k-x^*\|
        \leq \frac{2\nu^2}{9M},
    \]
    which means $y_k\in\mathcal{LQ}_g$.
\end{proof}

\subsection*{Lemma~\ref{lem.loglog vs log} and its proof}
\begin{lemma}\label{lem.loglog vs log}
    Let $\{a_i\}$ satisfy $a_i \le C\,a_{i-1}^2$ with $C>0$ and $0<Ca_0<1$. Let $\eta_0 = Ca_0$ and define
    \[
        \epsilon^* \;=\; \frac{1}{C} \Big(\eta_0\ln\!\frac{1}{\eta_0}\Big)^2.
    \]
    Let $i_{\min} = \min\{i\in\mathbb{N}:\ a_i\le \epsilon\}$. Then for all $0 < \epsilon \leq \epsilon^*$,
    \[
        i_{\min}\;\le\;\Big\lceil \log_2\!\frac{a_0}{\epsilon}\Big\rceil.
    \]
\end{lemma}

\begin{proof}
    From the quadratic recurrence, we have the bound
    \[
        a_i \;\le\; C^{2^{i}-1} a_0^{2^{i}} \;=\; \frac{(C a_0)^{2^{i}}}{C} \;=\; \frac{\eta_0^{2^{i}}}{C}\,.
    \]
    Thus, $a_i\le\epsilon$ is guaranteed if $\eta_0^{2^{i}}\le C\epsilon$, which is equivalent to
    \[
        2^{i} \;\ge\; \frac{\ln(1/(C\epsilon))}{\ln(1/\eta_0)}.
    \]
    It follows that
    \[
        i_{\min}\;\le\;\Big\lceil \log_2\!\Big(\frac{\ln(1/(C\epsilon))}{\ln(1/\eta_0)}\Big)\Big\rceil.
    \]
    To bound this by $\lceil \log_2(a_0/\epsilon) \rceil$, it suffices to show
    \[
        \frac{\ln(1/(C\epsilon))}{\ln(1/\eta_0)} \;\le\; \frac{a_0}{\epsilon} \quad\Longleftrightarrow\quad w\ln\!\frac{1}{w}\;\le\;\eta_0\ln\!\frac{1}{\eta_0},
    \]
    where $w := C\epsilon$. Let $K := \eta_0\ln(1/\eta_0)$. Since $\eta_0 \in (0, 1)$, $K \le 1/e$.

    For any $\epsilon \le \epsilon^*$, we have $w = C\epsilon \le K^2 \le (1/e)^2 < 1/e$. Because the function $x \mapsto x\ln(1/x)$ is strictly increasing on $(0, 1/e]$, evaluating it at $K^2$ yields
    \[
        w\ln\!\frac{1}{w} \;\le\; K^2 \ln\!\Big(\frac{1}{K^2}\Big) \;=\; 2K \Big(K\ln\!\frac{1}{K}\Big).
    \]
    By the maximum property of the function $x\ln(1/x)$, we have $K\ln(1/K) \le 1/e$. Substituting this gives
    \[
        w\ln\!\frac{1}{w} \;\le\; \frac{2}{e}K \;<\; K \;=\; \eta_0\ln\!\frac{1}{\eta_0}.
    \]
    This confirms the desired inequality.
\end{proof}
\section{Technical proofs in~\Cref{sec.second variant}}\label{sec.appendix 2}
\subsection*{Proof to \Cref{lem.growth of Ak}}\label{sec.app-c}
\begin{proof}
    First note Line~\ref{line.alg2-update-A} in \Cref{alg.ms accelerated utr}. For $i\geq 0$
    \begin{equation*}
        A_{i+1}^{1/2}-A_i^{1/2} = \frac{a_i}{A_{i+1}^{1/2}+A_i^{1/2}}=\frac{1}{A_{i+1}^{1/2}+A_i^{1/2}}\sqrt{\frac{A_{i+1}}{\sigma_i}}\geq \frac{1}{2\sqrt{\sigma_i}}.
    \end{equation*}
    The second equality comes from $\sigma_ia_i^2=A_i+a_i=A_{i+1}$~(\eqref{eq.update ak}). Summing up the above from $i=0$ to $k-1$ gives
    \begin{equation*}
        A_k \geq \frac{1}{4} \left (\sum_{i=0}^{k-1}\frac{1}{\sigma_i^{1/2}}\right)^2,
    \end{equation*}
    from \eqref{eq.relation lambda sigma}, the above gives
    \begin{equation}\label{eq.Ak helper 1}
        A_k \geq \frac{\eta}{4M} \left (\sum_{i=0}^{k-1} \frac{1}{\|d_i\|^{1/2}}\right)^2,
    \end{equation}
    on the other hand, from \eqref{eq.ms estimate advanced} we have
    \begin{equation*}
        B_k = \frac{3\gamma M}{8}\sum_{i=0}^{k-1} A_{i+1}\|d_i\|^3 \leq \frac{1}{2}\|v_0-x^*\|^2.
    \end{equation*}
    To estimate $A_k$ from below, define $\zeta_i = \|d_i\|^{1/2}$, $D = \frac{4}{3\gamma M}\|v_0-x^*\|^2$, we use the following auxiliary optimization problem
    \begin{equation*}
        \zeta^*=\min _{\zeta \in \mathbb{R}^k}\left\{\sum_{i=0}^{k-1} \frac{1}{\zeta_i}: \quad \sum_{i=0}^{k-1} A_{i+1} \zeta_i^6 \leq D\right\}.
    \end{equation*}
    Introducing the Lagrange multiplier \(\lambda\ge0\), the optimality condition gives
    \begin{equation*}
        -\frac{1}{\zeta_i^2}
        +
        6\lambda A_{i+1}\zeta_i^5
        =
        0,
        \qquad i=0,\ldots,k-1.
    \end{equation*}
    Equivalently, \(
    \frac{1}{\zeta_i^2}
    =
    6\lambda A_{i+1}\zeta_i^5. \) Letting \(w:=6\lambda\), we obtain
    \begin{equation*}
        \frac{1}{\zeta_i^2}
        =
        w A_{i+1}\zeta_i^5,
        \qquad i=0,\ldots,k-1.
    \end{equation*}
    Thus \(
    \zeta_i
    =
    \left(
    \frac{1}{wA_{i+1}}
    \right)^{1/7}. \) We have $w>0$ and the constraint is active,
    \begin{equation*}
        D=\sum_{i=0}^{k-1} A_{i+1}\left(\frac{1}{w A_{i+1}}\right)^{6 / 7}=\frac{1}{w^{6 / 7}} \sum_{i=0}^{k-1} A_{i+1}^{1 / 7},
    \end{equation*}
    therefore $\zeta^*=\sum_{i=0}^{k-1}\left(w A_{i+1}\right)^{1 / 7}=\frac{1}{D^{1 / 6}}\left(\sum_{i=0}^{k-1} A_{i+1}^{1 / 7}\right)^{7 / 6}$, plugging back $\|d_i\| $ and $ \|v_0-x^*\|^2$, we have
    \begin{equation*}
        A_k\geq\sum_{i=0}^{k-1}\frac{1}{\|d_i\|^{1/2}} \geq \left (\frac{3\gamma M}{4\|v_0-x^*\|^2}\right)^{1/6} \left(\sum_{i=0}^{k-1} A_{i+1}^{1 / 7}\right)^{7 / 6},
    \end{equation*}
    from \eqref{eq.Ak helper 1} we have
    \begin{equation}
        \label{eq.Ak helper 2}
        A_k \geq \frac{\eta}{4M}\left (\frac{3\gamma M}{4\|v_0-x^*\|^2}\right)^{1/3} \left(\sum_{i=1}^{k} A_{i}^{1 / 7}\right)^{7 / 3}, \ k\geq 1.
    \end{equation}
    Denote $\omega = \frac{\eta}{4M}\left (\frac{3\gamma M}{4\|v_0-x^*\|^2}\right)^{1/3}$, $C_k =\left(\sum_{i=1}^k A_i^{1 / 7}\right)^{2 / 3}$, plugging them into \eqref{eq.Ak helper 2} we have
    \begin{equation*}
        C_1\geq\omega^{1/7},\ C_{k+1}^{3 / 2}-C_k^{3 / 2} \geq \omega^{1 / 7} C_{k+1}^{1 / 2},
    \end{equation*}
    which gives
    \begin{equation*}
        \begin{aligned}
            \omega^{1 / 7} C_{k+1}^{1 / 2} & \leq\left(C_{k+1}^{1 / 2}-C_k^{1 / 2}\right)\left(C_{k+1}^{1 / 2}\left(C_{k+1}^{1 / 2}+C_k^{1 / 2}\right)+C_k\right)                                                                 \\
                                           & \leq\left(C_{k+1}^{1 / 2}-C_k^{1 / 2}\right)\left(C_{k+1}^{1 / 2}\left(C_{k+1}^{1 / 2}+C_k^{1 / 2}\right)+\frac{1}{2} C_{k+1}^{1 / 2}\left(C_{k+1}^{1 / 2}+C_k^{1 / 2}\right)\right) \\
                                           & =\frac{3}{2} C_{k+1}^{1 / 2}\left(C_{k+1}-C_k\right).
        \end{aligned}
    \end{equation*}
    Thus $C_k \geq \omega^{1 / 7}\left(1+\frac{2}{3}(k-1)\right), k \geq 1$. For $A_k$, by \eqref{eq.Ak helper 2} we have
    \begin{equation*}
        A_k  \geq \omega\left(C_k^{3 / 2}\right)^{7 / 3} \geq \omega\left(\omega^{1 / 7} \cdot \frac{2 k+1}{3}\right)^{7 / 2}=\omega^{3 / 2}\left(\frac{2 k+1}{3}\right)^{7 / 2}
        =\left(\frac{\eta}{4}\left (\frac{3\gamma}{4M^2\|v_0-x^*\|^2}\right)^{1/3}\right)^{3 / 2}\left(\frac{2 k+1}{3}\right)^{3.5}
    \end{equation*}
\end{proof}
\subsection*{Properties of $\psi(\sigma,y)$}
Now we introduce some basic properties of $\psi(\sigma,y)$:
\begin{lemma}
    \label{lem.psi property wrt sigma}
    For any $y\in\mathbb{R}^n$, suppose $0<\sigma\leq \Bar{\sigma}$, then we have
    \begin{equation}
        \label{eq.psi property wrt sigma}
        \left(\frac{\sigma}{\Bar{\sigma}}\right)^2 \psi(\sigma,y)\leq \psi(\Bar{\sigma},y) \leq \frac{\sigma}{\Bar{\sigma}}\psi(\sigma,y).
    \end{equation}
\end{lemma}
\begin{proof}
    Note that
    \begin{equation*}
        \sigma \psi(\sigma,y) = \left \| \left(\nabla^2 f(y)+\sigma I \right)^{-1}\nabla f(y)\right \|.
    \end{equation*}
    Since $\nabla^2 f(y)+\Bar{\sigma} I\succeq \nabla^2 f(y)+\sigma I \succ 0$, we have
    \begin{equation*}
        \sigma \psi(\sigma,y) \geq \Bar{\sigma} \psi(\Bar{\sigma},y),
    \end{equation*}
    which is the second argument. Similarly,
    \begin{equation*}
        \sigma^2 \psi(\sigma,y) = \left \| \left(\frac{\nabla^2 f(y)}{\sigma}+ I \right)^{-1}\nabla f(y)\right \|,
    \end{equation*}
    since $\frac{1}{\sigma}\nabla^2 f(y)+I\succeq \frac{1}{\Bar{\sigma}}\nabla^2 f(y)+I\succ I$, we have $\sigma^2 \psi(\sigma,y)\leq \Bar{\sigma}^2 \psi(\Bar{\sigma},y)$, which finished the proof.
\end{proof}
\begin{lemma}
    \label{lem.psi property wrt y}
    Suppose Assumption~\ref{assm.lipschitz} holds, for any $y,\Bar{y} \in \mathbb{R}^n$ and $\sigma>0$, then
    \begin{equation}
        \label{eq.psi property wrt y}
        \left \vert \psi(\sigma,y)-\psi(\sigma,\Bar{y}) \right\vert \leq \frac{1}{\sigma}\|y-\Bar{y}\| +\frac{M}{\sigma^2}\|y-\Bar{y}\|^2 +\frac{2M}{\sigma}\|y-\Bar{y}\|\delta,
    \end{equation}
    where $\delta:=\min\{\psi(\sigma,\Bar{y}),\psi(\sigma,y)\}$. Further, we have
    \begin{equation}
        \label{eq.psi property wrt y advanced}
        \psi(\sigma,y) \leq \frac{1}{\sigma}\|y-\Bar{y}\| +\frac{M}{\sigma^2}\|y-\Bar{y}\|^2 +\left(\frac{2M}{\sigma}\|y-\Bar{y}\|+1\right)\psi(\sigma,\Bar{y}).
    \end{equation}
\end{lemma}
\begin{proof}
    We denote
    \begin{gather*}
        x = \arg\min_{x\in \mathbb{R}^n} \ \nabla f(y)^T(x-y) +\frac{1}{2}(x-y)\nabla ^2 f(y)(x-y) +\frac{\sigma}{2}\|x-y\|^2,\\
        \Bar{x} = \arg\min_{x\in \mathbb{R}^n} \ \nabla f(\Bar{y})^T(x-\Bar{y}) +\frac{1}{2}(x-\Bar{y})\nabla ^2 f(\Bar{y})(x-\Bar{y}) +\frac{\sigma}{2}\|x-\Bar{y}\|^2,
    \end{gather*}
    the optimality conditions of the above are
    \begin{equation*}
        \nabla f(y) + \nabla ^2 f(y)(x-y) +\sigma(x-y)=0,\quad
        \nabla f(\Bar{y})+\nabla ^2 f(\Bar{y})(\Bar{x}-\Bar{y})+\sigma(\Bar{x}-\Bar{y})=0.
    \end{equation*}
    Denote
    \begin{equation*}
        v = \nabla f(y) + \nabla ^2 f(y)(x-y),\quad
        \Bar{v} =  \nabla f(\Bar{y})+\nabla ^2 f(\Bar{y})(\Bar{x}-\Bar{y}),\quad
        u = \nabla f(\Bar{y})+\nabla ^2 f(\Bar{y})(x-\Bar{y}).
    \end{equation*}
    Plugging them into the optimality conditions gives
    \begin{gather}
        \label{eq.def v}
        v+\sigma(x-y)=0,\\
        \label{eq.def bar v}
        \Bar{v}+\sigma(\Bar{x}-\Bar{y})=0,\\
        \label{eq.target distance initial}
        \left \vert \psi(\sigma,y)-\psi(\sigma,\Bar{y}) \right\vert= \left \vert\frac{1}{\sigma^2}\|v\| -\frac{1}{\sigma^2}\|\Bar{v}\|\right\vert\leq \frac{1}{\sigma^2}\|v-\Bar{v}\|.
    \end{gather}
    By \eqref{eq.def v}, we have
    \(
    v=-\sigma(x-y).
    \)
    Moreover, from the optimality condition,
    \(
    x-y=-(\nabla^2 f(y)+\sigma I)^{-1}\nabla f(y).
    \)
    Therefore, by the definition of \(\psi(\sigma,y)\),
    \(
    \sigma\psi(\sigma,y)
    =
    \|(\nabla^2 f(y)+\sigma I)^{-1}\nabla f(y)\|
    =
    \|x-y\|.
    \)
    Hence
    \(
    \|v\|
    =
    \sigma\|x-y\|
    =
    \sigma^2\psi(\sigma,y),
    \)
    or equivalently,
    \(
    \psi(\sigma,y)=\frac{\|v\|}{\sigma^2}.
    \)
    Similarly, \(\psi(\sigma,\bar y)=\frac{\|\bar v\|}{\sigma^2}. \)
    Consequently, we have \eqref{eq.target distance initial}.

    Subtracting \eqref{eq.def bar v} from \eqref{eq.def v}, we have
    \begin{equation*}
        \sigma(x-\Bar{x})+v -\Bar{v} = \sigma(y-\Bar{y}).
    \end{equation*}
    Plus both sides by $u$, we have
    \begin{equation*}
        \sigma(x-\Bar{x})+u-\Bar{v} = \sigma(y-\Bar{y})+u-v.
    \end{equation*}
    It is easy to show that
    \begin{equation*}
        \langle x-\Bar{x},u-\Bar{v}\rangle = \langle x-\Bar{x},\nabla ^2 f(\Bar{y})(x-\Bar{x})\rangle\geq 0,
    \end{equation*}
    thus from triangle inequality, we have
    \begin{equation*}
        \|u-\Bar{v}\| \leq \|\sigma(x-\Bar{x})+u-\Bar{v}\| \leq \sigma \|y-\Bar{y}\|+\|u-v\|,
    \end{equation*}
    hence
    \begin{equation}\label{eq.bound psi 1}
        \|v-\Bar{v}\| \leq \sigma\|y-\Bar{y}\| +2\|u-v\|.
    \end{equation}
    Now we bound $\|u-v\|$,
    \begin{equation}\label{eq.bound psi 2}
        \begin{aligned}
            \|u-v\| & = \|\nabla f(\Bar{y})+\nabla ^2 f(\Bar{y})(x-\Bar{y})-(\nabla f(y) + \nabla ^2 f(y)(x-y))\|                              \\
                    & = \|(\nabla f(\Bar{y})+\nabla^2 f(\Bar{y})(y-\Bar{y})-\nabla f(y))+(\nabla^2 f(\Bar{y})(x-y)-\nabla^2 f(y)(x-y))\|       \\
                    & \leq \|(\nabla f(\Bar{y})+\nabla^2 f(\Bar{y})(y-\Bar{y})-\nabla f(y))\| +\|(\nabla^2 f(\Bar{y})-\nabla^2 f(y))(x-y)\|    \\
                    & \leq \frac{M}{2}\|y-\Bar{y}\|^2+M\|y-\Bar{y}\|\|x-y\|  =\frac{M}{2}\|y-\Bar{y}\|^2+\sigma M \|y-\Bar{y}\|\psi(\sigma,y).
        \end{aligned}
    \end{equation}
    Plugging \eqref{eq.bound psi 1} and \eqref{eq.bound psi 2} into \eqref{eq.target distance initial}, we have
    \begin{equation*}
        \left \vert \psi(\sigma,y)-\psi(\sigma,\Bar{y}) \right\vert \leq \frac{1}{\sigma}\|y-\Bar{y}\| +\frac{M}{\sigma^2}\|y-\Bar{y}\|^2 +\frac{2M}{\sigma}\|y-\Bar{y}\|\psi(\sigma,y).
    \end{equation*}
    Similarly, we can prove
    \begin{equation*}
        \left \vert \psi(\sigma,y)-\psi(\sigma,\Bar{y}) \right\vert \leq \frac{1}{\sigma}\|y-\Bar{y}\| +\frac{M}{\sigma^2}\|y-\Bar{y}\|^2 +\frac{2M}{\sigma}\|y-\Bar{y}\|\psi(\sigma,\Bar{y}).
    \end{equation*}
    Hence we have proved \eqref{eq.psi property wrt y}, by applying triangle inequality, we have \eqref{eq.psi property wrt y advanced}.
\end{proof}
\subsection*{Proof to \Cref{lem.reduce line search}}
\begin{proof}
    If \eqref{eq.bracket small lambda} holds, we have
    \begin{equation*}
        \|d\| = \|(\nabla^2 f(y(\sigma))+\sigma I)^{-1}\nabla f(y(\sigma))\|<\frac{\eta}{M}\sigma,
    \end{equation*}
    which is $\psi(\sigma,y(\sigma))<\frac{\eta}{M}$.
    Else if \eqref{eq.bracket big lambda} holds, we have
    \begin{equation*}
        \|d\| = \|(\nabla^2 f(y(\sigma))+\sigma I+\lambda I)^{-1}\nabla f(y(\sigma))\|=\frac{1}{M}\sigma.
    \end{equation*}
    By \eqref{eq.psi property wrt sigma}, we have
    \begin{equation*}
        \|(\nabla^2 f(y(\sigma))+\sigma I)^{-1}\nabla f(y(\sigma))\|> \|(\nabla^2 f(y(\sigma))+\sigma I+\lambda I)^{-1}\nabla f(y(\sigma))\|=\frac{1}{M}\sigma,
    \end{equation*}
    which means $\psi (\sigma,y(\sigma))>\frac{1}{M}$.
\end{proof}

\subsection*{Proof to \Cref{lem.complexity bisection}}
In fact, we prove a more detailed version of \Cref{lem.complexity bisection} here, which is \cref{coro.bisection bound} here. First, we introduce the following property of $y(\cdot)$, since it is an auxiliary lemma, and the proof is almost the same as in \citet{monteiro2013accelerated}.
\begin{lemma}[Lemma 7.13, \citet{monteiro2013accelerated}]
    \label{lem.curve condition}
    Suppose Assumption~\ref{assm.solvable} holds, and there exists $M_0>0$ such that
    \begin{equation}\label{eq.bounded assm point}
        \|x-x^*\|\leq M_0, \ \|v-x^*\|\leq M_0,
    \end{equation}
    then the curve $y(\cdot)$ satisfies
    \begin{equation}
        \label{eq.curve condition}
        \|y(s)-y(t)\| \leq \frac{2M_0}{t}(s-t), \ \forall s\geq t >0.
    \end{equation}
\end{lemma}
\begin{proof}
    We have
    \begin{equation*}
        y(\sigma) = \frac{A}{A+a(\sigma)}x+\frac{a(\sigma)}{A+a(\sigma)}v = x + \tau (\sigma) (v-x),
    \end{equation*}
    where $\tau(\sigma) = \frac{a(\sigma)}{A+a(\sigma)}.$ For any $s\geq t>0$,
    \begin{equation*}
        \|y(s)-y(t)\| = \left|\tau(s)-\tau(t)\right| \|v-x\|,
    \end{equation*}
    by the mean value theorem, we have
    \begin{equation*}
        \|y(s)-y(t)\| = \left|\tau'(\xi)\right|(s-t) \|v-x\|\leq M_0\left|\tau'(\xi)\right|(s-t),
    \end{equation*}
    where $\xi \in \left[t,s\right]$. Note that from \eqref{eq.update ak} we have $\sigma a^2=a+A$, it leads to
    \begin{equation*}
        \tau(\sigma) = \frac{a(\sigma)}{A+a(\sigma)} = \frac{1}{\sigma a(\sigma)} = \frac{2}{1+\sqrt{1+4A\sigma}}.
    \end{equation*}
    Its derivative is
    \begin{equation*}
        \tau'(\sigma) = -\frac{4A}{\sqrt{1+4A\sigma}\left(1+\sqrt{1+4A\sigma}\right)^2},
    \end{equation*}
    therefore for all $\sigma>0$, we have
    \begin{equation*}
        \left|\tau'(\sigma)\right| \leq \frac{4A}{\left(1+\sqrt{1+4A\sigma}\right)^2}  \leq \frac{4A}{4A\sigma}  = \frac{1}{\sigma}.
    \end{equation*}
    Therefore we have
    \begin{equation*}
        \begin{aligned}
            \|y(s)-y(t)\| & \leq \left|\tau'(\xi)\right|(s-t)\|x-v\|                          \\
                          & \leq \left|\tau'(\xi)\right|(s-t)\left(\|x-x^*\|+\|v-x^*\|\right) \\
                          & \leq 2M_0\left|\tau'(\xi)\right|(s-t)                             \\
                          & \leq \frac{2M_0}{t}(s-t).
        \end{aligned}
    \end{equation*}
\end{proof}
Next, we proceed to introduce the lemma with analyzes the difference between $\psi_-$ and $\psi_+$:
\begin{lemma}
    \label{lem.bound gap bisection}
    Suppose Assumption~\ref{assm.lipschitz} and Assumption~\ref{assm.solvable} hold, and there exists $M_0>0$ such that \eqref{eq.bounded assm point} holds. Then in \textnormal{R\&B}~(\Cref{alg.bisection ms}), we have
    \begin{equation}\label{eq.bound gap bisection}
        \begin{aligned}
            \psi_- - \psi_+
            \leq\; & \left( \frac{\sigma_+}{\sigma_-} \right)^2 \Bigg[
                \frac{1}{\sigma_+} \|y(\sigma_-)-y(\sigma_+)\|
            + \frac{M}{\sigma_+^2} \|y(\sigma_-)-y(\sigma_+)\|^2                          \\
                   & \quad + \frac{2M}{\sigma_+} \|y(\sigma_-)-y(\sigma_+)\| \cdot \psi_+
                \Bigg]
            + \left[ \left( \frac{\sigma_+}{\sigma_-} \right)^2 - 1 \right] \psi_+.
        \end{aligned}
    \end{equation}
    Further, it gives
    \begin{equation}
        \label{eq.bounded gap bisection sigma}
        \psi_- - \psi_+
        \leq
        \frac{\sigma_+}{\sigma_-^2}
        \Bigg[
            \frac{2M_0}{\sigma_-}
            +
            \frac{4MM_0^2}{\sigma_-^2}
            +
            2\left(
            \frac{2MM_0}{\sigma_-}+1
            \right)\psi_+
            \Bigg]
        (\sigma_+-\sigma_-).
    \end{equation}
\end{lemma}
\begin{proof}
    \begin{equation*}
        \begin{aligned}
            \psi_- -\psi_+
             & = \psi(\sigma_-,y(\sigma_-))-\psi(\sigma_+,y(\sigma_+))                                               \\
             & \leq \left(\frac{\sigma_+}{\sigma_-}\right)^2 \psi(\sigma_+,y(\sigma_-)) - \psi(\sigma_+,y(\sigma_+)) \\
             & \leq \left(\frac{\sigma_+}{\sigma_-}\right)^2 \Bigg[
                \frac{1}{\sigma_+}\|y(\sigma_-)-y(\sigma_+)\|
                + \frac{M}{\sigma_+^2} \|y(\sigma_-)-y(\sigma_+)\|^2 + \left( \frac{2M}{\sigma_+} \|y(\sigma_-)-y(\sigma_+)\| + 1 \right)\psi_+
            \Bigg] - \psi_+                                                                                          \\
             & = \left(\frac{\sigma_+}{\sigma_-} \right)^2 \Bigg[
                \frac{1}{\sigma_+} \|y(\sigma_-)-y(\sigma_+)\|
                + \frac{M}{\sigma_+^2} \|y(\sigma_-)-y(\sigma_+)\|^2 + \frac{2M}{\sigma_+} \|y(\sigma_-)-y(\sigma_+)\|\psi_+
            \Bigg]                                                                                                   \\
             & \quad + \left[ \left(\frac{\sigma_+}{\sigma_-} \right)^2 - 1 \right] \psi_+.
        \end{aligned}
    \end{equation*}
    The second line is because of \Cref{lem.psi property wrt sigma}, the third line is because of \Cref{lem.psi property wrt y}. To derive \eqref{eq.bounded gap bisection sigma}, plug
    \eqref{eq.curve condition} into \eqref{eq.bound gap bisection} and use
    \[
        \sigma_+-\sigma_-\le \sigma_+,
        \qquad
        \left(\frac{\sigma_+}{\sigma_-}\right)^2-1
        =
        \frac{(\sigma_+-\sigma_-)(\sigma_++\sigma_-)}{\sigma_-^2}
        \le
        \frac{2\sigma_+}{\sigma_-^2}(\sigma_+-\sigma_-).
    \]
\end{proof}
Now we can proceed with the main proof.
\begin{lemma}
    \label{coro.bisection bound}
    Suppose Assumption~\ref{assm.lipschitz} and
    Assumption~\ref{assm.solvable} hold, and there exists \(M_0>0\) such that
    \eqref{eq.bounded assm point} holds. Then in \textnormal{R\&B}, we have
    \begin{equation}
        \label{eq.bisection-coefficient-bound}
        \frac{\sigma_+}{\sigma_-^2}
        \left[
            \frac{2M_0}{\sigma_-}
            +
            \frac{4MM_0^2}{\sigma_-^2}
            +
            2\left(
            \frac{2MM_0}{\sigma_-}+1
            \right)\psi_+
            \right] \leq
        C(G_0,M_0,M,\epsilon_g),
    \end{equation}
    where
    \begin{equation}
        \label{eq.definition C bisection}
        C(G_0,M_0,M,\epsilon_g)
        :=
        \sqrt{G_0}\Bigg(
        \frac{(1+2\theta)^3M_0}
            {\sqrt{2}M\eta^2\epsilon_g^{3/2}}
        +
        \frac{(1+2\theta)^2M_0^2}
            {\sqrt{M}\eta\epsilon_g^2}
        +
        \frac{\sqrt{2}(1+2\theta)^{3/2}M_0}
            {M\epsilon_g^{3/2}}
        +
        \frac{(1+2\theta)^2}
            {M^{3/2}\eta^{3/2}\epsilon_g}
        \Bigg).
    \end{equation}
    If the bisection procedure does not terminate, then
    \begin{equation}
        \label{eq.bound value by brackets}
        \sigma_+-\sigma_-
        >
        \frac{1-\eta}
        {M C(G_0,M_0,M,\epsilon_g)}.
    \end{equation}
    Further, the number of oracle calls during the bisection is bounded by
    \begin{equation}
        \label{eq.bisection-oracle-bound}
        \log\left(
        \frac{
                M\sqrt{\frac{MG_0}{\eta}}\,
                C(G_0,M_0,M,\epsilon_g)
            }{
                1-\eta
            }
        \right).
    \end{equation}
\end{lemma}
\begin{proof}
    Let
    \(
    \underline{\sigma}
    :=
    \sqrt{\frac{2M\epsilon_g}{1+2\theta}},
    \overline{\sigma}
    :=
    \sqrt{\frac{MG_0}{\eta}}.
    \)
    By the construction of the bracket points in \eqref{eq.bracket points},
    throughout the bisection we have
    \(
    \underline{\sigma}
    \le
    \sigma_-
    \le
    \sigma_+
    \le
    \overline{\sigma}.
    \)
    Moreover, the upper bracket satisfies
    \(
    \psi_+\le \frac{\eta}{M}.
    \)
    Using \eqref{eq.bounded gap bisection sigma}, we get
    \[
        \psi_- - \psi_+
        \le
        B_k(\sigma_+-\sigma_-),
    \]
    where
    \[
        B_k
        :=
        \frac{\sigma_+}{\sigma_-^2}
        \left[
            \frac{2M_0}{\sigma_-}
            +
            \frac{4MM_0^2}{\sigma_-^2}
            +
            2\left(
            \frac{2MM_0}{\sigma_-}+1
            \right)\psi_+
            \right].
    \]
    We now upper bound \(B_k\). First,
    \[
        \frac{2\sigma_+M_0}{\sigma_-^3}
        \le
        \frac{
            2\sqrt{\frac{MG_0}{\eta}}M_0
        }{
            \left(
            \sqrt{\frac{2M\epsilon_g}{1+2\theta}}
            \right)^3
        }  =
        \frac{
            \sqrt{G_0}(1+2\theta)^{3/2}M_0
        }{
            \sqrt{2}M\sqrt{\eta}\epsilon_g^{3/2}
        } \le
        \frac{
            \sqrt{G_0}(1+2\theta)^3M_0
        }{
            \sqrt{2}M\eta^2\epsilon_g^{3/2}
        }.
    \]
    Second,
    \[
        \frac{4\sigma_+MM_0^2}{\sigma_-^4}
        \le
        \frac{
            4\sqrt{\frac{MG_0}{\eta}}MM_0^2
        }{
            \left(
            \sqrt{\frac{2M\epsilon_g}{1+2\theta}}
            \right)^4
        } =
        \frac{
            \sqrt{G_0}(1+2\theta)^2M_0^2
        }{
            \sqrt{M}\sqrt{\eta}\epsilon_g^2
        } \le
        \frac{
            \sqrt{G_0}(1+2\theta)^2M_0^2
        }{
            \sqrt{M}\eta\epsilon_g^2
        }.
    \]
    Third, using \(\psi_+\le \eta/M\),
    \[
        \frac{4\sigma_+MM_0\psi_+}{\sigma_-^3}
        \le
        \frac{4\eta M_0\sigma_+}{\sigma_-^3} \le
        \frac{
            \sqrt{2G_0}(1+2\theta)^{3/2}M_0\sqrt{\eta}
        }{
            M\epsilon_g^{3/2}
        } \le
        \frac{
            \sqrt{2G_0}(1+2\theta)^{3/2}M_0
        }{
            M\epsilon_g^{3/2}
        }.
    \]
    Finally,
    \[
        \frac{2\sigma_+\psi_+}{\sigma_-^2}
        \le
        \frac{2\eta\sigma_+}{M\sigma_-^2} \le
        \frac{
            \sqrt{G_0}(1+2\theta)\sqrt{\eta}
        }{
            M^{3/2}\epsilon_g
        } \le
        \frac{
            \sqrt{G_0}(1+2\theta)^2
        }{
            M^{3/2}\eta^{3/2}\epsilon_g
        }.
    \]
    Combining the above four estimates gives
    \[
        B_k
        \le
        C(G_0,M_0,M,\epsilon_g),
    \]
    where \(C(G_0,M_0,M,\epsilon_g)\) is defined in
    \eqref{eq.definition C bisection}. Hence
    \[
        \psi_- - \psi_+
        \le
        C(G_0,M_0,M,\epsilon_g)(\sigma_+-\sigma_-).
    \]

    If the bisection procedure does not terminate, then by the bracket condition
    we have
    \(
    \psi_- - \psi_+
    >
    \frac{1-\eta}{M}.
    \)
    Therefore,
    \[
        \sigma_+-\sigma_-
        >
        \frac{1-\eta}
        {M C(G_0,M_0,M,\epsilon_g)}.
    \]
    This proves \eqref{eq.bound value by brackets}.

    By the mechanism of bisection, the total number of bisection steps in the
    \(k\)-th iteration, denoted by \(N_k\), satisfies
    \[
        N_k
        \le
        \log\left(
        \frac{
                \left(
                \sqrt{\frac{MG_0}{\eta}}
                -
                \sqrt{\frac{2M\epsilon_g}{1+2\theta}}
                \right)
                M C(G_0,M_0,M,\epsilon_g)
            }{
                1-\eta
            }
        \right)
        \le
        \log\left(
        \frac{
                M\sqrt{\frac{MG_0}{\eta}}\,
                C(G_0,M_0,M,\epsilon_g)
            }{
                1-\eta
            }
        \right).
    \]
    Since \(G_0\) and \(M_0\) are polynomially bounded in \(D_0\), by the
    definition of \(C(G_0,M_0,M,\epsilon_g)\), and omitting fixed algorithmic
    parameters, we conclude
    \(
    N_k
    =
    O\left(
    \log\frac{MD_0}{\epsilon_g}
    \right).
    \)
\end{proof}

\subsection*{Proof to \Cref{lem.bounded iterate}}
\begin{proof}
    First, we show that \eqref{eq.ub yk sigma} is a direct consequence of \eqref{eq.bounded iterate y}:
    \begin{equation*}
        \|\nabla f(y_k(\sigma)) - \nabla f(x^*) -\nabla^2 f(x^*)(y_k(\sigma)-x^*)\| \leq \frac{M}{2}\|y_k(\sigma)-x^*\|^2.
    \end{equation*}
    Using \eqref{eq.first-order exp}, by the triangle inequality, we have
    \begin{equation*}
        \begin{aligned}
            \|\nabla f(y_k(\sigma)) \| & \leq \|\nabla f(x^*)+\nabla^2 f(x^*)(y_k(\sigma)-x^*)\| +\frac{M}{2}\|y_k(\sigma)-x^*\|^2                                    \\
                                       & \leq \|\nabla^2 f(x^*)(y_k(\sigma)-x^*)\|+\frac{M}{2}\|y_k(\sigma)-x^*\|^2                                                   \\
                                       & \leq \left(\frac{4}{\sqrt{3\gamma}}+1\right)\|\nabla ^2 f^*\|D_0+\frac{M}{2}\left(\frac{4}{\sqrt{3\gamma}}+1\right)^2 D_0^2.
        \end{aligned}
    \end{equation*}
    Therefore, we only need to prove \eqref{eq.bounded iterate} and \eqref{eq.bounded iterate y} hold. We prove them by induction. It is trivial that they hold for $i=0$. Suppose that they hold for $i=k$. We will prove that they also hold for $i=k+1$. Note that
    \begin{equation*}
        \begin{aligned}
            \|x_{k+1}-x^*\| & \leq \|x_{k+1}-y_k\| + \|y_k-x^*\|                                            \\
                            & \leq_{(a)} \|d_k\|+ \frac{A_k}{A_{k+1}} \|x_k-x^*\| +\frac{a_k}{A_{k+1}}\|v_k-x^*\| \\
                            & \leq_{(b)} \frac{1}{A_{k+1}}\left(A_{k+1}\|d_k\| +A_k\|x_k-x^*\| +  a_k D_0\right)  \\
                            & \leq_{(c)} \frac{1}{A_{k+1}}\left (\sum_{i=0}^k A_{i+1}\|d_i\|\right)+D_0.
        \end{aligned}
    \end{equation*}
    $(a)$ is from \eqref{eq.update yk}, $(b)$ is from \eqref{eq.ms estimate advanced}, and $(c)$ is derived by iterating
    \begin{equation*}
        A_{k+1} \|x_{k+1}-x^*\|\leq A_{k+1}\|d_k\|+A_k\|x_k-x^*\|+a_k D_0,
    \end{equation*}
    which is derived by multiplying both sides of the third line by $A_{k+1}$.

    Summing up \eqref{eq.helper bounded x}, we have
    \begin{equation*}
        \frac{3\gamma}{8}\sum_{i=0}^k A_{i+1}\sigma_i \|d_i\|^2 \leq \frac{1}{2}D_0^2,
    \end{equation*}
    to bound $\sum_{i=0}^k A_{i+1}\|d_i\|$, we come to the optimization problem:
    \begin{equation*}
        \max _{\zeta \in \mathbb{R}^{k+1}_+}\left\{\sum_{i=0}^{k} A_{i+1}\zeta_i: \quad \sum_{i=0}^{k} A_{i+1} \sigma_i\zeta_i^2 \leq \frac{4}{3\gamma}D_0^2\right\},
    \end{equation*}
    through similar analysis in \citet[Lemma A.2]{monteiro2013accelerated} we have
    \begin{equation*}
        \sum_{i=0}^k A_{i+1}\|d_i\| \leq \sqrt{\frac{4}{3\gamma}}\cdot \sqrt{\sum_{i=0}^k \frac{A_{i+1}}{\sigma_i}} D_0,
    \end{equation*}
    therefore,
    \begin{equation}\label{eq.bound initial x}
        \|x_{k+1}-x^*\|  \leq \frac{1}{A_{k+1}}\left (\sum_{i=0}^k A_{i+1}\|d_i\|\right)+D_0  \leq_{(a)} \left(\sqrt{\frac{4}{3\gamma}}\cdot \sum_{i=0}^k\sqrt{\frac{1}{\sigma_i}}\cdot \sqrt{\frac{1}{A_{k+1}}} +1\right)D_0.
    \end{equation}
    $(a)$ is from the fact that $A_k$ is monotone and $2$-norm is majorized by $1$-norm. From \eqref{eq.update ak} we have
    \(
    a_k \geq \frac{1}{2\sigma_k}+\sqrt{\frac{A_k}{\sigma_k}},
    \)
    hence
    \begin{equation*}
        A_{k+1}\geq A_k+\frac{1}{2\sigma_k}+\sqrt{\frac{A_k}{\sigma_k}}\geq A_k+\frac{1}{4\sigma_k}+\sqrt{\frac{A_k}{\sigma_k}}.
    \end{equation*}
    Taking the square root of both sides,
    \begin{equation*}
        \sqrt{A_{k+1}}\geq \sqrt{A_k}+\frac{1}{2\sqrt{\sigma_k}},
    \end{equation*}
    Iterating the above inequality gives
    \begin{equation*}
        \sqrt{A_{k+1}}\geq \sum_{i=0}^k \frac{1}{2\sqrt{\sigma_i}}.
    \end{equation*}
    Plugging the above into \eqref{eq.bound initial x}, we have
    \begin{equation*}
        \|x_{k+1}-x^*\| \leq \left( \frac{4}{\sqrt{3\gamma}}+1\right) D_0.
    \end{equation*}
    For $v_{k+1}$, we have \eqref{eq.bounded iterate}, \eqref{eq.bounded iterate y} and \eqref{eq.ub yk sigma} hold for iteration $k$. As a result, the bisection search procedure is valid and \eqref{eq.relation lambda sigma} holds for the $k$-th iteration due to Corollary~\ref{coro.bisection bound}. Hence $\|v_{k+1}-x^*\|\leq D_0$.

    To prove the boundedness of $y_{k+1}(\sigma)$, just note that $y_{k+1}(\sigma)$ is a linear combination of $v_{k+1}$ and $x_{k+1}$.
\end{proof}


\clearpage

\tableofcontents

\clearpage

\end{document}